\documentclass[11pt,reqno]{amsart}
\usepackage[numbers, square]{natbib}
\usepackage{mathptmx}
\usepackage{amsmath}
\usepackage{amscd}
\usepackage{amssymb}
\usepackage{amsthm}
\usepackage{enumerate}
\usepackage{xspace}
\usepackage[all,tips]{xy}
\usepackage[dvips]{graphicx}
\usepackage{verbatim}
\usepackage{syntonly}
\usepackage{hyperref}
\usepackage{amsmath, amsthm, graphics, amssymb,fullpage,color, epsfig,url}
\usepackage{indentfirst}
\usepackage{color}

\newtheorem{thm}{Theorem}
\newtheorem{lemma}{Lemma}
\newtheorem{defn}{Definition}
\newtheorem{rem}{Remark}
\newtheorem{cor}{Corollary}
\newtheorem{prop}{Proposition} 

\newcommand{\disp}{\displaystyle}
\DeclareMathOperator{\dist}{dist}
\DeclareMathOperator{\supp}{supp}
\DeclareMathOperator{\dv}{div}
\DeclareMathOperator{\Vol}{Vol}
\newcommand{\eps}{\varepsilon}
\newcommand{\vp}{\varphi}
\newcommand{\al}{\alpha}
\newcommand{\be}{\beta}
\newcommand{\ga}{\gamma}
\newcommand{\de}{\delta}
\newcommand{\Ga}{\Gamma}
\newcommand{\Si}{\Sigma}
\newcommand{\te}{\theta}
\newcommand{\La}{\Lambda}
\newcommand{\Om}{\Omega}
\newcommand{\si}{\sigma}
\newcommand{\iny}{\infty}
\newcommand{\del}{ \partial}
\newcommand{\su}{\subset}
\newcommand{\LP}{\Delta}
\newcommand{\gr}{\nabla}
\newcommand{\norm}[1]{\left\vert \left\vert #1\right\vert\right\vert}
\newcommand{\innp}[1]{\left< #1 \right>}
\newcommand{\abs}[1]{\left\vert#1\right\vert}
\newcommand{\set}[1]{\left\{#1\right\}}
\newcommand{\brac}[1]{\left[#1\right]}
\newcommand{\pr}[1]{\left( #1 \right) }
\newcommand{\der}[2]{\frac{\del #1}{\del #2} }
\newcommand{\dert}[2]{\frac{\del^2 #1}{\del #2 ^2} }
\newcommand{\derm}[3]{\frac{\del^2 #1}{\del #2 \del #3} }
\renewcommand{\d}[1]{\ensuremath{\operatorname{d}\!{#1}}}
\newcommand{\N}{\ensuremath{\mathbb{N}}}
\newcommand{\R}{\ensuremath{\mathbb{R}}}
\newcommand{\Z}{\ensuremath{\mathbb{Z}}}

\begin{document} 

\title{Parabolic theory as a high-dimensional limit of elliptic theory}

\author[Davey]{Blair Davey}
\address{Department of Mathematics, City College of New York CUNY, New York, NY 10031, USA}
\email{bdavey@ccny.cuny.edu}


\maketitle

\begin{abstract}
The aim of this article is to show how certain parabolic theorems follow from their elliptic counterparts.
This technique is demonstrated through new proofs of five important theorems in parabolic unique continuation and the regularity theory of parabolic equations and geometric flows.
Specifically, we give new proofs of an $L^2$ Carleman estimate for the heat operator, and the monotonicity formulas for the frequency function associated to the heat operator, the two-phase free boundary problem, the flow of harmonic maps, and the mean curvature flow.
The proofs rely only on the underlying elliptic theorems and limiting procedures belonging essentially to probability theory.
In particular, each parabolic theorem is proved by taking a high-dimensional limit of the related elliptic result. \\



\end{abstract}

\section{Introduction}
\label{Intro}

Experts have long realized the parallels between elliptic and parabolic theory of partial differential equations.
It is well-known that elliptic theory may be considered a static, or steady-state, version of parabolic theory.  
And in particular, if a parabolic estimate holds, then by eliminating the time parameter, one immediately arrives at the underlying elliptic statement.  
Producing a parabolic statement from an elliptic statement is not as straightforward.  
In this article, we demonstrate a method for producing parabolic theorems from their elliptic analogues.  
Specifically, we show that certain parabolic estimates may be obtained by taking high-dimensional limits of the corresponding elliptic result.

The idea to consider parabolic theory as a high-dimensional limit of elliptic theory was used by Perelman as a motivation for introducing what is now known as the {\em Perelman reduced volume}, \cite{P02} Section 6.
The methods of the proof, as well as the general philosophy that parabolic theory is a high-dimensional limit of elliptic theory, are discussed in the blog of Tao, \cite{TaoB}.
Our set-up will be simpler than that of the Ricci flow, and we will be able to use a form of classical probabilistic formulae, essentially going back to Wiener \cite{W23}, with a slight modification used by Sverak in \cite{Sv11}.

The method of obtaining parabolic theorems by taking high-dimensional limits is demonstrated through five new proofs.
The first is a proof of an $L^2$ Carleman estimate for the operator $\LP + \del_t$.
This Carleman estimate was proved by Escauriaza in \cite{E00} and Tataru in \cite{Ta99}, with further analysis by Koch and Tataru in \cite{KT01}.
The second new proof, originally proved by Poon in \cite{P96}, shows that the frequency function associated to the heat equation is monotonically non-decreasing.
These two theorems, motived by their elliptic counterparts, allowed the authors of \cite{E00}, \cite{P96} and \cite{Ta99} to use the established techniques for elliptic theory to prove that strong unique continuation also holds for solutions to the heat equation.
This was a major step forward in the theory of unique continuation for parabolic equations.
The third new proof is of a monotonicity formula for two-phase free boundary parabolic problems.
This formula was proved in \cite{Caf93} by Caffarelli, and extended by Caffarelli and Kenig in \cite{CK98} to prove regularity of solutions to parabolic equations and their singular perturbations.
The fourth new proof in this article is of a monotonicity formula for the flow of harmonic maps.
The original proof is due to Struwe, \cite{St88} (and many other proofs since).
And the fifth new proof is of a monotonicity formula for mean curvature flow, which was proved by Huisken in \cite{Hu90}.
These two theorems were crucial in the development of regularity theory for geometric flows.
The parabolic theorems mentioned here were discovered independently, but we show that they in fact follow from their elliptic counterparts in a common way.
The starting point of each new proof is a classical formula used in probability together with a related calculation from \cite{Sv11}.

The author hopes that the techniques presented in this article may find other applications.
In particular, if a certain elliptic result is known to hold in every dimension, then it may be possible to prove the corresponding parabolic result using the ideas presented here.

The article is organized as follows.  
In Section \ref{MT}, we develop the connection between the elliptic and parabolic theory by presenting Wiener's calculation from \cite{W23} and its variant presented by Sverak in \cite{Sv11}.  
Section \ref{PG} contains a collection of calculations and statements that will be referred to throughout the article as well as some more detailed remarks on the proof philosophy.
The $L^2$ parabolic Carleman estimate is proved in Section \ref{CE}.
The frequency function theorem for the heat operator is presented in Section \ref{FF}.
Section \ref{FBP} contains the monotonicity formulae for two-phase free boundary problems.
In Section \ref{HM}, harmonic maps are introduced and the monotonicity formula is stated and proved.
The results for minimal surfaces and mean curvature flow are given in Section \ref{MM}.

\section{Measure Theoretic Details}
\label{MT}

Within this section, we establish the two main tools of this article, Lemmas \ref{PFTS} and \ref{PFT}.  
In all subsequent sections, these lemmas allow us to pass from a known elliptic notion to the corresponding parabolic result.

We start with some classical ideas concerning random walks, going back to Wiener \cite{W23}.
An explanation of these standard ideas is also available in Sverak's notes \cite{Sv11}.
Consider $d$ particles, each one moving randomly in one spatial dimension.  
Let $x_1, x_2, \ldots, x_d$ denote the coordinates of these particles.
Rather than imposing a condition on the step size, we instead require the more universal condition that if each $x_i$ makes $n$ random steps, denoted $y_{i,1}, y_{i,2}, \ldots, y_{i,n}$,  then for some fixed $t > 0$
\begin{equation}
\abs{y}^2 = \sum_{i=1}^d \brac{y_{i,1}^2 + \ldots y_{i,n}^2} = 2 d t.
\label{yNorm}
\end{equation}
Assuming that each $x_i$ starts at the origin, after these $n$ steps, the new positions will be
\begin{equation}
x_i = y_{i,1} + y_{i,2} + \ldots + y_{i,n}.
\label{xiDef}
\end{equation}
To understand the probability law for the events $\pr{y_{1,1}, \ldots, y_{1,n}, \ldots, y_{d,1}, \ldots, y_{d,n}}$, assume that the vectors $\pr{y_{1,1}, \ldots, y_{1,n}, \ldots, y_{d,1}, \ldots, y_{d,n}}$ are distributed over the $n\cdot d-1$ dimensional sphere of radius $\sqrt{2 d t}$ uniformly with respect to the canonical surface measure.  
If the surface measure is normalized to have total measure equal to $1$, then this surface measure, $\mu_{n,d}^t$, is given by
$$\mu_{n,d}^t = \frac{1}{\abs{S^{n \cdot d-1}} \pr{2 d t}^{\frac{n \cdot d-1}{2}}} \si_{n \cdot d-1}^t,$$
where $ \si_{n \cdot d-1}^t$ denotes the canonical surface measure of the sphere described by equation \eqref{yNorm}. \\

Define a function
$$f_{n,d} : \R^{n \cdot d} \to \R^d$$
by
\begin{align}
& f_{n,d}\pr{y_{1,1}, \ldots, y_{1,n}, \ldots, y_{d,1}, \ldots, y_{d,n}} \nonumber \\
&= \pr{y_{1,1} + \ldots + y_{1,n}, \ldots, y_{d,1} + \ldots + y_{d,n}}.
\label{fndDef}
\end{align}
In other words, for each $i = 1, \ldots, d$, equation \eqref{xiDef} holds.

We need to compute the push-forward of $\mu_{n,d}^t$ by $f_{n,d}$, denoted by $\nu_{n,d}^t = f_{n,d \#} \mu_{n,d}^t$.  
For simplicity, we may replace $f_{n,d}$ above with
\begin{align*}
& \tilde f_{n,d}\pr{y_{1,1}, y_{1,2}, \ldots, y_{1,n}, \ldots, y_{d,1}, y_{d,2}, \ldots, y_{d,n}} \\
&= \pr{\sqrt{n}\; y_{1,1}, \sqrt{n} \; y_{2,1}, \ldots, \sqrt{n} \; y_{d,1}},
\end{align*}
since the two maps are related by an orthogonal transformation that leaves the measure unchanged.  Therefore, we write $x_i = \sqrt{n} \; y_{i,1}$ in what follows.  
The push-forward is computed in two steps.  
First, we push-forward the measure $\mu_{n,d}^t$ by the projection 
$$\pr{y_{1,1}, y_{1,2}, \ldots, y_{1,n}, \ldots, y_{d,1}, y_{d,2}, \ldots, y_{d,n}} \mapsto  \pr{y_{1,1}, y_{2,1}, \ldots, y_{d,1}},$$ 
then we dilate by a factor of $\sqrt{n}$.
A computation shows that the projection gives
$$ \frac{1 }{\abs{S^{n \cdot d-1}} \pr{2 d t}^{\frac{n \cdot d-1}{2}}}  \abs{S^{n \cdot d-1 - d}} \pr{2 d t}^{\frac{n \cdot d-1 - d}{2}}\pr{1 - \frac{y_{1,1}^2 + \ldots + y_{d,1}^2}{2 d t}}^{\frac{n \cdot d - d - 2}{2}} \d{y},$$
where $\d{y} = \d{y}_{1,1} \ldots \d{y}_{d,1}.$
Using $x_i = \sqrt{n} \; y_{i,1}$, we see that
\begin{align*}
\nu_{n,d}^t 
&=  \frac{\abs{S^{n \cdot d-1 - d}} }{\abs{S^{n \cdot d-1}} \pr{2 n d t}^{\frac{d}{2}}} \pr{1 - \frac{x_1^2 + \ldots + x_{d}^2}{2 n d t}}^{\frac{n \cdot d - d - 2}{2}} \d{x}_1 \ldots \d{x}_{d} \\
&=  \frac{\abs{S^{n \cdot d-1 - d}} }{\abs{S^{n \cdot d-1}} \pr{2 n d t}^{\frac{d}{2}}} \pr{1 - \frac{\abs{x}^2}{2 n d t}}^{\frac{n \cdot d - d - 2}{2}} \d{x},
\end{align*}
where $\d{x} = \d{x}_1 \d{x}_2 \ldots \d{x}_d$.

The following sets will be used repeatedly throughout the article and are related to one another through the function $f_{n,d}$. 

\begin{defn}
Let $S_t^n$ denote the sphere of radius $\sqrt{2 d t}$ in $\R^{n\cdot d}$,
\begin{equation}
{S}_t^n = \set{y \in \R^{n \cdot d} : \abs{y} = \sqrt{2 d t}}.
\label{StnDefn}
\end{equation} 
Let $B_{nt}$ denote the ball of radius $\sqrt{2 n d t}$ in $\R^{d}$,
\begin{equation}
{B}_{nt} = \set{x \in \R^{d} : \abs{x} \le \sqrt{2 n d t}}.
\label{BntDefn}
\end{equation} 
\end{defn}

\begin{rem} 
At this point, we notice that the expression for $\nu_{n,d}^t $ is not necessarily well-defined when the argument is negative, or when $2 n d t < \abs{x}^2$.
But notice that by \eqref{xiDef}, standard inequalities, and \eqref{yNorm},
\begin{align*}
\abs{x}^2 = \sum_{i=1}^d x_i^2 = \sum_{i=1}^d \pr{y_{i,1} + \ldots + y_{i,n}}^2 \le n \sum_{i=1}^d \pr{y_{i,1}^2 + \ldots + y_{i,n}^2} = n \cdot 2d t.
\end{align*}
Thus, the argument is always non-negative and the expression is well-defined for all $n \in \N$.
In fact, we have that $f_{n,d}\pr{S_t^n} = B_{nt}$ and $\nu_{n,d}^t$ is a measure supported on $B_{nt}$.
\end{rem}

There is a nice relationship between the integrability of a function with respect to the measures $\nu_{n,d}^t$ and a Gaussian measure.

\begin{lemma}
Define 
\begin{align}
& G_{t, n}\pr{x} = G_n\pr{x,t} := \frac{\abs{S^{n \cdot d-1 - d}} }{\abs{S^{n \cdot d-1}} \pr{2 n d t}^{\frac{d}{2}}} \pr{1 - \frac{\abs{x}^2}{2 n d t}}^{\frac{n \cdot d - d - 2}{2}} \chi_{B_{nt}},
\label{GnDefn} \\
&  G_t\pr{x} = G\pr{x,t} := \pr{\frac 1 {4 \pi t}}^{\frac d 2} \exp\pr{- \frac{\abs{x}^2}{4t}},
\label{GDefn}
\end{align}
where $\chi_{B_{nt}}$ is the indicator function of the set $B_{nt}$.
If $\vp : \R^d \to \R$ is integrable with respect to the Gaussian measure $G_t\pr{x} \d{x}$, then $\vp$ is also integrable with respect to $G_{t,n}\pr{x} \d{x}$ for every $n \in \N$. 
\label{intLemma}
\end{lemma}

\begin{proof}
Since 
\begin{align*}
\frac{G_{t,n}\pr{x}}{G_t\pr{x}}
&= \frac{\abs{S^{n \cdot d-1 - d}} }{\abs{S^{n \cdot d-1}} } \pr{\frac{ 4 \pi}{ 2 n d}}^{\frac{d}{2}} \pr{1 - \frac{2}{nd}\frac{\abs{x}^2}{4 t}}^{\frac{n \cdot d - d - 2}{2}}  \exp\pr{ \frac{\abs{x}^2}{4t}} \chi_{B_{nt}}
\end{align*}
and $B_{nt} = \set{ x \in \R^d : \frac{\abs{x}^2}{4t} \le \frac{nd}{2}}$, then this ratio is bounded and positive.
In fact, the maximum occurs whenever $\disp \frac{\abs{x}^2}{4t} = \frac d 2 +1$, so that $\disp \norm{\frac{G_{t,n}\pr{x}}{G_t\pr{x}}}_{L^\iny\pr{\R^d}} = \mathcal{C}_{n,d}$, where
\begin{align*}
\mathcal{C}_{n,d}
&:= \frac{\abs{S^{n \cdot d-1 - d}} }{\abs{S^{n \cdot d-1}} } \pr{\frac{ 4 \pi}{ 2 n d}}^{\frac{d}{2}} \brac{1 - \frac{2}{nd}\pr{\frac d 2 +1}}^{\frac{n \cdot d - d - 2}{2}}  \exp\pr{ \frac d 2 + 1}.
\end{align*}
Therefore,
\begin{align*}
\int_{\R^d} \vp\pr{x} G_{t,n}\pr{x} \d{x}
&
\le \mathcal{C}_{n,d} \int_{\R^d} \vp\pr{x} G_t\pr{x} \d{x}
< \iny,
\end{align*}
since $\vp$ is integrable with respect to $G_t\pr{x} \d{x}$.
\end{proof}

Using the definition of push-forward in combination with Lemma \ref{intLemma}, we arrive at the following classical result.

\begin{lemma}
Let $G_{t,n}\pr{x}$ and $G_t\pr{x}$ be as defined in \eqref{GnDefn} and \eqref{GDefn}, respectively.
If $\vp: \R^d \to \R$ is integrable with respect to $G_{t}\pr{x} \d{x}$, then for every $n \in \N$,
\begin{align*}
& \frac{1}{\abs{S_t^n}} \int_{S^n_t} \vp\pr{f_{n,d}\pr{y}} \si_{n \cdot d-1}^t
= \int_{\R^d} \vp\pr{x} G_{t,n}\pr{x} \d{x}.
\end{align*}
\label{PFTS}
\end{lemma}

Following Sverak in \cite{Sv11}, we now broaden this viewpoint so that $t$ is a parameter instead of a fixed constant.  
That is, we think of the measures $\nu_{n,d}^t$ 
as time slices of a space-time object that comes from projections of some global measure $\mu_{n,d}$ in the space $y\in \R^{n\cdot d}$ (not just the spheres) onto the space-time $\pr{x, t} \in \R^d \times \R_+$ (not just time slices).
To do this, define a function
$$F_{n,d} : \R^{n \cdot d} \to \R^d \times \R_+$$
by
\begin{align}
& F_{n,d}\pr{y_{1,1}, \ldots, y_{1,n}, \ldots, y_{d,1}, \ldots, y_{d,n}} \nonumber \\
&= \pr{y_{1,1} + \ldots + y_{1,n}, \ldots, y_{d,1} + \ldots + y_{d,n}, \frac{\abs{y}^2}{2d}}.
\label{FndDef}
\end{align}
In other words, $F_{n,d}$ is defined so that \eqref{yNorm} and \eqref{xiDef} both hold.

The global measure $\mu_{n,d}$ on $\R^{n \cdot d}$ has the property that 
$$F_{n,d \#} \pr{\mu_{n,d}} = \int_0^\iny f_{n,d \#} \mu_{n,d}^t \d{t} = \int_0^\iny \nu_{n,d}^t \d{t}.$$
We see that
$$\mu_{n,d} = \frac{1}{d \abs{S^{n \cdot d -1}} \abs{y}^{n \cdot d -2}} \d{y}.$$

This viewpoint gives us another pair of sets, these ones related through $F_{n,d}$.

\begin{defn}
Let $B_\tau^n$ denote the ball of radius $\sqrt{2 d \tau}$ in $\R^{n\cdot d}$,
\begin{equation}
{B}_\tau^n = \set{y \in \R^{n \cdot d} : \abs{y} \le \sqrt{2 d \tau}}.
\label{BtnDefn}
\end{equation} 
Let $K_{n\tau}$ denote the following space-time cone in $\R^{d} \times \R_+$,
\begin{equation}
K_{n\tau} = \set{\pr{x,t} \in \R^d \times \R_+: x \in B_{nt}, t \le \tau}.
\label{KntDefn}
\end{equation} 
\end{defn}

\begin{rem}
It follows from the previous remark that $F_{n,d}\pr{B_\tau^n} = K_{n\tau}.$
Consequently, $\mu_{n,d}$ is a measure on space-time cones.
\end{rem}

Not surprisingly, there is a version of the integrability relationships for this setting as well.

\begin{lemma}
Let $G_n\pr{x,t}$, $G\pr{x, t}$ be as given in \eqref{GnDefn} and \eqref{GDefn}, respectively.
If $\phi : \R^d \times \pr{0, T} \to \R$ is integrable with respect to $G\pr{x,t} \d{x}\,\d{t}$, then $\phi$ is also integrable with respect to $G_n\pr{x,t} \d{x} \, \d{t}$ for every $n$.
\label{intLemma2}
\end{lemma}

The proof of Lemma \ref{intLemma2} mirrors that of Lemma \ref{intLemma}, so we omit it.
By the definition of the pushforward, the computations from above, and the previous lemma, we reach the following result.

\begin{lemma}
\label{PFT}
Let $G_n\pr{x,t}$ be as given in \eqref{GnDefn}.
If $\phi: \R^d \times \pr{0, T} \to \R$ is integrable with respect to $G\pr{x, t} \d{x} \, \d{t}$, then for any $\tau \le T$,
\begin{align*}
& \frac{1}{d \abs{S^{n \cdot d -1}}} \int_{B_\tau^n} \phi\pr{F_{n,d}\pr{y}} \abs{y}^{2 - n \cdot d} \d{y} 
=  \int_0^\tau \int_{\R^d} \phi\pr{x,t} G_n\pr{x,t} \d{x} \, \d{t}.
\end{align*}
Moreover, if $T = \iny$, then
\begin{align*}
&\frac{1}{d \abs{S^{n \cdot d -1}}} \int_{\R^{n \cdot d}} \phi\pr{F_{n,d}\pr{y}} \abs{y}^{2 - n \cdot d } \d{y} 
= \int_0^\iny \int_{\R^d} \phi\pr{x,t}G_n\pr{x,t} \d{x} \, \d{t}.
\end{align*}
\end{lemma}

\section{Preliminaries}
\label{PG}

Here we collect the additional tools that will be used repeatedly throughout the article.
From now on, we use the following convention:
Each function $v = v\pr{y}$ can be thought of as a solution to some (possibly non-homogeneous) elliptic equation; while every $u = u\pr{x,t}$ can be thought of as a solution a (possibly non-homogeneous) parabolic equation.
We relate $u$ and $v$ to one another through $F_{n,d}$, that is, $v\pr{y} = u\pr{F_{n,d}\pr{y}} = u\pr{x,t}$.

First, we state a lemma that relates the derivatives of $u$ and $v$, whenever $u$ and $v$ satisfy the relation $v\pr{y} = u\pr{F_{n,d}\pr{y}} = u\pr{x,t}$.
As we see below, if $u$ satisfies a parabolic partial differential equation, then $v$ is a solution to a related (possibly non-homogeneous) elliptic equation.
Therefore, in combination with Lemma \ref{PFTS} or \ref{PFT}, these relations build the bridge between the elliptic theory and the parabolic theory.
This lemma will be referred to throughout the article.
The proof of each statement follows from an application of the chain rule.

\begin{lemma}
Let $u: \R^{d} \times \pr{0, T} \to \R$.  
If $v: \R^{n \cdot d} \to \R$ is such that $v\pr{y} = u\pr{F_{n,d}\pr{y}}$, then the following hold:
\begin{align}
&\der{v}{y_{i,j}}  = \der{u}{x_i} +  \frac{y_{i,j}}{d} \der{u}{t} 
\label{firstDer} \\
&\derm{v}{y_{i,j}}{y_{k,l}} =\derm{u}{x_i}{x_k} + \frac{y_{k,l}}{d} \derm{u}{x_i}{t} +\frac{ y_{i,j}}{d} \derm{u}{x_k}{t} + \frac{y_{i,j} y_{k,l}}{d^2} \dert{u}{t} \nonumber \\
&\qquad\qquad\; + \frac{\de_{ik} \de_{jl}}{d} \der{u}{t} 
\label{mixedDer} \\
& \LP v = n\pr{\LP u + \der{u}{t}} + \frac{2}{d} \pr{x, t} \cdot \gr_{\pr{x,t}}\pr{ \der{u}{t}}
\label{Laplace} \\
& y \cdot \gr v = x \cdot \gr u +2t \der{u}{t}
\label{vngr} \\
& \abs{\gr v}^2 = n \abs{\gr u}^2 + \frac{2}{d} \brac{ {\pr{x,t} \cdot \gr_{\pr{x,t}} u} }\der{u}{t}
\label{grSq}
\end{align}
\label{ChainR}
\end{lemma}

In \eqref{GnDefn}, we introduced the functions $G_n\pr{x,t} = G_{t,n}\pr{x}$ which serve as the weights in the measures that come up in Lemmas \ref{PFTS} and \ref{PFT}.
Considered as a sequence, these functions converge to the standard Gaussian function.

\begin{lemma}
Let $G_n\pr{x,t}$, $G\pr{x, t}$ be as given in \eqref{GnDefn} and \eqref{GDefn}, respectively. 
For every $\pr{x,t} \in \R^d \times \R_+$, $\disp \lim_{n \to \iny} G_n\pr{x,t} = G\pr{x,t}$.
\label{GnLimit}
\end{lemma}

\begin{proof}
By Stirling's formula
$$\lim_{n \to \iny}\pr{2 n d}^{-\frac{d}{2}} \frac{\abs{S^{n \cdot d-1 - d}}}{\abs{S^{n \cdot d -1}}}   = \pr{\frac{1}{4\pi}}^{d/2}$$
and by standard limit laws,
\begin{equation}
\lim_{n \to \iny} \pr{1 - \frac{\abs{x}^2}{2 n d t}}^{\frac{n \cdot d - d - 2}{2}} = \exp\pr{-\frac{\abs{x}^2}{4t}}.
\label{expLim}
\end{equation}
Since $\disp \lim_{n \to \iny} B_{nt} = \R^d$, then $\disp \lim_{n \to \iny} G_{n}\pr{x,t} = G\pr{x,t}$.
\end{proof}

Having established all of our main tools, we now describe the main technique that will be used below to prove each parabolic theorem from an elliptic counterpart.
Given a parabolic function $u = u\pr{x,t}$, we define $v_n = v_n\pr{y}$ so that $v_n\pr{y} = u\pr{F_{n,d}\pr{y}}$ for every $n \in \N$.
Using the relations presented in Lemma \ref{ChainR}, we show that each $v_n$ is elliptic in the sense that it solves a related time-indepedent equation.
Therefore, there is an elliptic theorem that applies to each $v_n$.  
By either Lemma \ref{PFTS} or Lemma \ref{PFT}, an integral involving $v_n$ over a sphere or a ball is equivalent to some integral involving $u$ over a time-slice or a space-time cylinder.
Once we establish this relationship for every $n \in \N$, we take a limit as $n \to \iny$ and employ Lemma \ref{GnLimit} to reach the conclusion of the parabolic theorem.

By examining \eqref{Laplace}, we see that there is not an exact connection between elliptic and parabolic equations through $F_{n,d}$ in the following sense: 
If $u$ solves a homogeneous parabolic equation, then $v_n$ solves a possibly non-homogeneous elliptic equation.
Therefore, to prove a parabolic theorem using a high-dimensional limit argument, we may require a non-homogeneous version of the related elliptic theorem.
In fact, to prove each of the parabolic monotonicity theorems presented in Sections \ref{FF}, \ref{FBP}, \ref{HM}, and \ref{MM}, we employ non-homogeneous elliptic theorems.
These new elliptic theorems resemble their homogeneous counterparts and are proved using the same techniques in the current article.

\section{Carleman Estimates}
\label{CE}

Within this section, we use an elliptic Carleman estimate to prove its parabolic analogue.  
The main tool used in this proof is Lemma \ref{PFT}.

The following elliptic Carleman estimate is the $L^2$ case of Theorem 1 from \cite{ABG81}.
The original theorem was used to establish unique continuation properties of functions that satisfy $\abs{\LP v} \le\abs{V} \abs{v}$, for $v \in H^{2,q}_{loc}\pr{\Om}$, $V \in L^w_{loc}\pr{\Om}$, where $w > \frac{N}{2}$, and $\Om \subset \R^N$ is open and connected.

\begin{thm}[\cite{ABG81}, Theorem 1]
For any $\ga \in \R$ and all $v \in H^{2,2}_c\pr{\R^N \setminus \set{0}}$, the following inequality holds
\begin{equation}
 \norm{\abs{y}^{-\ga +2} \LP v}_{L^2\pr{\R^N}} \ge c\pr{\ga, N} \norm{\abs{y}^{-\ga} v}_{L^2\pr{\R^N}},
\end{equation}
where
$$c\pr{\ga, N} = \inf_{\ell \in \Z_{\ge 0}} \abs{\pr{\frac{N}{2} + \ell + \ga - 2}\pr{\frac{N}{2} + \ell - \ga}}.$$
\label{ECE}
\end{thm}

\begin{rem}
In order for this theorem to be meaningful, we must ensure that $c\pr{\ga, N} > 0$.  
\end{rem}

The following parabolic Carleman estimate is the $L^2$ version of Theorem 1 from \cite{E00}.
The original theorem was used to prove strong unique continuation of solutions to the heat equation.  

\begin{thm}[\cite{E00}, Theorem 1]
Let $d \ge 1$.  Let $\al \in \R$ be such that $\be = 2\al - \frac{d}{2} - 1 > 0$ is not an integer.  Then there is a constant $C$ depending only on $d$ and $\eps = \dist\pr{\be, \Z_{\ge 0}}$ such that the inequality
\begin{align*}
\int_0^\iny \int _{\R^d} t^{-2\al} e^{-\frac{\abs{x}^2}{4t}} \abs{u}^2 \d{x} \, \d{t} 
\le C\pr{d, \eps} \int_0^\iny \int _{\R^d} t^{-2\al + 2} e^{-\frac{\abs{x}^2}{4t}} \abs{\LP u + \del_t u}^2 \d{x} \, \d{t},
\end{align*}
holds for every $u \in C^\iny_0\pr{\R^{d+1}_+ \setminus \set{\pr{0,0}}}$.
\label{PCE}
\end{thm}

We now show that Theorem \ref{PCE} follows from the elliptic result, Theorem \ref{ECE}, Lemma \ref{PFT}, and the results of Section \ref{PG}.

\begin{proof}
Let $u \in C^\iny_0\pr{\R^{d+1}_+ \setminus \set{\pr{0,0}}}$.
For every $n \in \N$, let $v_{n} : \R^{n \cdot d} \to \R$ satisfy
$$v_{n}\pr{y} = u\pr{F_{n,d}\pr{y}}.$$
Since $u \in C^\iny_0$, then $v_n$ and $\LP v_n$ satisfy the hypotheses of Lemma \ref{PFT}, then for $\ga_n$ to be defined below
\begin{align}
\int_{\R^{n \cdot d}} &\abs{v_{n}\pr{y}}^2 \abs{y}^{-2 \ga_n} \d{y}
= \int_{\R^{n \cdot d}} \abs{ u\pr{F_{n,d}\pr{y}}}^2 \abs{y}^{n\cdot d - 2 -2 \ga_n} \abs{y}^{2 - n \cdot d} \d{y}
\nonumber \\
&= d \abs{S^{n \cdot d-1}} \int_0^\iny \int_{\R^d} \abs{u\pr{x, t}}^2 \pr{2dt}^{\frac{n \cdot d}{2} -1 -\ga_n} G_n\pr{x,t} \d{x} \, \d{t}
\label{uEst}
\end{align}
and
\begin{align}
& \frac{1}{n^2 d \abs{S^{n \cdot d-1}}  } \int_{\R^{n \cdot d}} \abs{\LP v_n\pr{y}}^2 \abs{y}^{4-2 \ga_n} \d{y} \label{uLPEst} \\
&= \int_0^\iny \int_{\R^d} \abs{{\LP u + \del_{t} u} + \frac{2\brac{\pr{x,t} \cdot \gr_{\pr{x,t}}\del_t u}}{n \cdot d}}^2 \pr{2dt}^{\frac{n \cdot d}{2}+1- \ga_n} G_n\pr{x,t} \d{x} \, \d{t},\nonumber
\end{align}
where we have used \eqref{Laplace} from Lemma \ref{ChainR}.
By Theorem \ref{ECE} 
\begin{align}
\int_{\R^{n \cdot d}} \abs{\LP v_n\pr{y}}^2 \abs{y}^{4-2 \ga_n} \d{y} 
&\ge C\pr{\ga_n, n d}^2 \int_{\R^{n \cdot d}} \abs{v_n\pr{y}}^2 \abs{y}^{-2 \ga_n} \d{y} .
\label{T1App}
\end{align}
Combining \eqref{uEst}, \eqref{uLPEst}, and \eqref{T1App} and simplifying, we see that
\begin{align*}
& \frac{C\pr{\ga_n, n d}^2}{4 n^2 d^2} \int_0^\iny \int_{\R^d} \abs{u\pr{x, t}}^2 t^{\frac{n \cdot d}{2} -1 -\ga_n} G_n\pr{x,t} \d{x} \, \d{t} \\
&\le  \int_0^\iny \int_{\R^d} \abs{{\LP u + \del_{t} u} + \frac{2}{n \cdot d}\brac{\pr{x,t} \cdot \gr_{\pr{x,t}}\del_t u}}^2 t^{\frac{n \cdot d}{2}+1- \ga_n} G_n\pr{x,t} \d{x} \, \d{t} .
\end{align*}
Setting $2\al = \ga_n - \frac{n\cdot d - d - 2}{2}$ and simplifying gives
\begin{align*}
& \frac{C\pr{\ga_n, n d}^2}{8 n^2 d^2 } \int_0^\iny \int_{\R^d} \abs{u\pr{x, t}}^2 t^{ \frac{d }{2} - 2\al} G_n\pr{x,t} \d{x} \, \d{t} \\
&\le  \int_0^\iny \int_{\R^d} \abs{{\LP u + \del_{t} u} }^2 t^{\frac{d }{2} - 2\al+2} G_n\pr{x,t} \d{x} \, \d{t} \\
&+ \frac{4}{n^2 d^2} \int_0^\iny \int_{\R^d} \abs{\pr{x,t} \cdot \gr_{\pr{x,t}}\del_t u }^2 t^{\frac{d }{2} - 2\al+2} G_n\pr{x,t} \d{x} \, \d{t} .
\end{align*}
Since
\begin{align*}
c\pr{\ga_n, nd} 
&= \inf_{\ell \in \Z_{\ge 0}} \abs{\brac{ nd -2 + \ell + \pr{2\al - \frac{ d }{2} - 1}}\brac{ \ell - \pr{2\al - \frac{ d }{2} - 1}}} \\
&= \inf_{\ell \in \Z_{\ge 0}} \abs{\pr{ nd -2 + \ell +\be}\pr{ \ell -\be}}
\ge \pr{nd - 2 + \be} \eps,
\end{align*}
then it follows that
\begin{align*}
& \int_0^\iny \int_{\R^d} \abs{u\pr{x, t}}^2 t^{ \frac{d }{2} - 2\al} G_n\pr{x,t} \d{x} \, \d{t} \\
&\le \frac{8 n^2 d^2 }{\pr{nd - 2+\be}^2 \eps^2} \int_0^\iny \int_{\R^d} \abs{{\LP u + \del_{t} u} }^2 t^{\frac{d }{2} - 2\al+2} G_n\pr{x,t} \d{x} \, \d{t} \\
&+ \frac{32 }{\pr{nd - 2+\be}^2 \eps^2}\int_0^\iny \int_{\R^d} \abs{\pr{x,t} \cdot \gr_{\pr{x,t}}\del_t u }^2 t^{\frac{d }{2} - 2\al+2} G_n\pr{x,t} \d{x} \, \d{t} .
\end{align*}
We now take the limit as $n \to \iny$. 
By an application of Lemma \ref{GnLimit}, we see that
\begin{align*}
& \int_0^\iny \int_{\R^{d}} \abs{u\pr{x, t}}^2 t^{-2\al} e^{-\frac{\abs{x}^2}{4t}} \d{x}\, \d{t}  \\
&\le C\pr{d, \eps} \int_0^\iny \int_{\R^{d}} \abs{\LP u + \del_{t} u}^2 t^{-2\al + 2} e^{-\frac{\abs{x}^2}{4t}} \d{x} \, \d{t} ,
\end{align*}
as required.
\end{proof}

\section{Frequency Functions}
\label{FF}

In this section, we explore the non-trivial connection between frequency functions for solutions to elliptic and parabolic equations.  
In particular, we use a monotonicity result for solutions to the Poisson equation in conjunction with Lemmas \ref{PFTS} and \ref{PFT} to prove the corresponding monotonicity result for solutions to the heat equation.

In \cite{GL86} and \cite{GL87}, Garofalo and Lin studied the properties of frequency functions and used their results to prove a strong unique continuation theorem for solutions to elliptic partial differential equations.  
To do this, they generalized the following result due to Almgren from \cite{Al79} for frequency functions associated to harmonic functions.

\begin{thm}
For $v : \R^N \to \R$, define
\begin{align*}
& H\pr{r; v} = \int_{\del B_r} \abs{v\pr{y}}^2 \d{S}\pr{y} \\ 
& D\pr{r; v} = \int_{B_r} \abs{\gr v\pr{y}}^2 \d{y} \\
& L\pr{r; v} = \frac{r D\pr{r; v}}{ H\pr{r; v}}.
\end{align*}
If $\LP v = 0$ in $\R^N$, then $L\pr{r; v}$ is monotonically non-decreasing in $r$.
\label{eFF}
\end{thm}

In what follows, we require a non-homogeneous version of Theorem \ref{eFF} to prove the parabolic analogue.

\begin{cor}
For $v : \R^N \to \R$, define $H$, $D$, and $L$ as in the statement of Theorem \ref{eFF}.
If $\LP v = h$ in $\R^N$, where $h$ is bounded and measurable, then 
\begin{align*}
L^\prime\pr{r; v} 
&\ge 2 \frac{\pr{ \int_{\del B_r} v \, {y \cdot \gr v} \d{S}\pr{y} } \pr{ \int_{B_r} h \, v \, \d{y}} }{ \pr{\int_{\del B_r} \abs{v\pr{y}}^2 \d{S}\pr{y}}^2}
- 2 \frac{ \int_{B_r} h \pr{ y \cdot \gr v } \d{y} }{\int_{\del B_r} \abs{v\pr{y}}^2 \d{S}\pr{y} }.
\end{align*}
\label{eFF2}
\end{cor}

The proof of this result uses the classical techniques, but we include it here for completeness.
For brevity, we at times drop the $v$ in the notation for $H, D$, and $L$ when it is understood that these functions are associated to $v$.

\begin{proof}
A computation shows that
\begin{align*}
H^\prime\pr{r} &= \frac{N-1}{r} H\pr{r} + \frac{2}{r} \int_{\del B_r} v \, y \cdot \gr v \, \d{S}\pr{y}.
\end{align*}
Notice that
\begin{align}
D\pr{r} &= \int_{B_r} \abs{\gr v}^2 \d{y}
= \int_{B_r} \brac{ \frac 1 2 \LP\pr{v^2} - v \LP v } \d{y} \nonumber \\
&= \frac{1}{r} \int_{\del B_r} v \, y \cdot \gr v \, \d{S}\pr{y} 
-\int_{B_r} h \, v \, \d{y},
\label{IExp}
\end{align}
where we used that $\LP v = h$ and integration by parts.
Now we compute the derivative of $D\pr{r}$.
\begin{align*}
D^\prime\pr{r} &= \int_{\del B_r}  \abs{\gr v}^2 \d{S}\pr{y}
= \frac{1}{r} \int_{\del B_r} \innp{y\abs{\gr v}^2, \frac y r} \d{S}\pr{y}.
\end{align*}
For each $i = 1, 2, \ldots, N$, an integration by parts shows that
\begin{align*}
\int_{\del B_r} &y_i \abs{\gr v}^2 \cdot \frac {y_i} r \d{S}\pr{y}
= \int_{B_r} \del_i\pr{y_i \abs{\gr v}^2} \, \d{y} \\
&= \int_{B_r} \abs{\gr v}^2 \, \d{y}
+ 2\sum_{j = 1}^N \int_{B_r} y_i \der{v}{y_j} \derm{v}{y_i}{y_j} \, \d{y} \\
&= \int_{B_r} \abs{\gr v}^2  \d{y}
- 2\int_{B_r} \der{v}{y_i} \der{v}{y_{i}} \, \d{y}
- 2\sum_{j = 1}^N \int_{B_r} \dert{v}{y_j} \, \der{v}{y_i} \, y_i \, \d{y} \\
&+ 2r \sum_{j = 1}^N \int_{\del B_r}  \der{v}{y_i} \frac{y_i }{r} \der{v}{y_j} \frac{y_j}{r} \d{S}\pr{y}.
\end{align*}
Since $\LP v = h$, then
\begin{align*}
D^\prime\pr{r} 
&= \int_{\del B_r} \abs{\gr v}^2 \d{S}\pr{y} \\
&= \frac{N-2}{r}\int_{B_r} \abs{\gr v}^2  \d{y}
- \frac 2 r \int_{B_r} h \pr{y \cdot \gr v} \d{y}
+ \frac 2{r^2} \int_{\del B_r} \pr{y \cdot \gr v}^2 \d{S}\pr{y} \\
&= \frac{N-2}{r} D\pr{r}
+ \frac{2}{r^2} \int_{\del B_r} \pr{y \cdot \gr v}^2 \d{S}\pr{y}
- \frac 2 r \int_{B_r} h \pr{ y \cdot \gr v} \d{y}.
\end{align*}
Combining our computations,
\begin{align*}
L^\prime\pr{r}
&= \frac{2}{r} \set{\frac{ {\int_{\del B_r} \pr{y \cdot \gr v}^2 \d{S}\pr{y}} }{{\int_{\del B_r} \abs{v\pr{y}}^2 \d{S}\pr{y}}} 
-  \brac{\frac{ \int_{\del B_r} v \, y \cdot \gr v \, \d{S}\pr{y} }{{\int_{\del B_r} \abs{v\pr{y}}^2 \d{S}\pr{y}}}}^2} \\
&+ 2 \frac{ \pr{\int_{\del B_r} v \, y \cdot \gr v \, \d{S}\pr{y}}\pr{\int_{B_r} h \, v \, \d{y}}}{\pr{\int_{\del B_r} \abs{v\pr{y}}^2 \d{S}\pr{y}}^2}
- 2 \frac{\int_{B_r} h \pr{ y \cdot \gr v} \d{y}}{\int_{\del B_r} \abs{v\pr{y}}^2 \d{S}\pr{y}}
\end{align*}
By Cauchy-Schwarz,
$$ \pr{ \int_{\del B_r} v \, y \cdot \gr v \, \d{S}\pr{y} }^2 \le \pr{ \int_{\del B_r} \pr{y \cdot \gr v}^2 \,\d{S}\pr{y}} \pr{ \int_{\del B_r} \abs{v}^2 \, \d{S}\pr{y}},$$
so the first term is non-negative and the conclusion of the corollary follows.
\end{proof}

We use this non-homogeneous elliptic result to reprove the parabolic version from \cite{P96}, restated using the notation from \cite{E00}.
This result was a crucial tool in the proof of strong unique continuation of the heat equation.

\begin{thm}[\cite{P96}]
Let $u: \R^d \times \pr{0,T} \to \R$ and let $G_t\pr{x}$ be as in \eqref{GDefn}.
Define
\begin{align*}
& \mathcal{H}\pr{t; u} = \int_{\R^d} \abs{u\pr{x,t}}^2 G_t\pr{x} \d{x} \\ 
& \mathcal{D}\pr{t; u} = \int_{\R^d} \abs{\gr u\pr{x,t}}^2 G_t\pr{x} \d{x} \\
& \mathcal{L}\pr{t; u} = \frac{t \mathcal{D}\pr{t; r}}{ \mathcal{H}\pr{t; r}}.
\end{align*}
If $\disp \LP u + \der{ u}{t} = 0$ in $\R^d \times \pr{0, T}$, then $\mathcal{L}\pr{t; u}$ is monotonically non-decreasing in $t$.
\label{pFF}
\end{thm}

The non-homogeneous version of Almgren's frequency function, along with the tools developed in the early part of this article, will be used to prove Theorem \ref{pFF}.

\begin{proof}
Let $u: \R^{d} \times \pr{0,T} \to \R$ be a solution to $\disp \LP u + \del_t u = 0$ in $\R^d \times \pr{0,T}$.
For every $n \in \N$, let $v_n : B_T^n \su \R^{n \cdot d} \to \R$ satisfy
$$v_n\pr{y} = u\pr{F_{n,d}\pr{y}}.$$
Then by \eqref{Laplace} from Lemma \ref{ChainR},
\begin{align*}
\LP v_n 
&= n\pr{\LP u + \del_t {u} } + \frac{2}{d} \pr{x, t} \cdot \gr_{\pr{x,t}}\pr{ \del_t {u} } 
= \frac{2}{d} \pr{x, t} \cdot \gr_{\pr{x,t}}\pr{ \del_t {u} } 
=:  J\pr{x,t}.
\end{align*}
For every $n$, define $h_{n} : B_T^n \to \R$ so that 
$$h_n\pr{y} = J\pr{F_{n,d}\pr{y}}$$
and then
$$\LP v_n = h_n.$$
Thus, we may apply Corollary \ref{eFF2} to $v_n$ on any ball of radius $\sqrt{2 d t}$ for $t < T$.

First we compute the frequency function associated to $v_n$ on the ball of radius $\sqrt{2 d t}$.
By Lemma \ref{PFTS},
\begin{align*}
H\pr{\sqrt{2 d t}, v_n} 
&= \int_{S_t^n} \abs{v_n\pr{y}}^2 \si_{n \cdot d -1}^t \\
&= \pr{2 d t}^{\frac{n \cdot d -1}{2}} \abs{S^{n \cdot d -1}} \int_{\R^d} \abs{u\pr{x,t}}^2 G_{t,n}\pr{x} \d{x}.
\end{align*}
Using the expression \eqref{IExp} along with Lemma \ref{PFTS} and Lemma \ref{PFT},
\begin{align*}
&D\pr{\sqrt{2 d t}, v_n}
= \pr{2 d t}^{ -\frac{1}{2}} \int_{S_t^n} v_n\pr{y} \pr{y \cdot \gr v_n\pr{y}} \si_{n \cdot d -1}^t \\
&- \int_{B_t^n} h_n\pr{y} v_n\pr{y} \abs{y}^{n \cdot d -2} \abs{y}^{2- n \cdot d } \d{y} \\
&= \pr{2 d t}^{ \frac{n \cdot d -2}{2}} \abs{S^{n \cdot d -1}} \int_{\R^d} u\pr{x,t} \pr{x \cdot \gr u +2t \der{u}{t}} G_{t, n}\pr{x} \d{x} \\
& - 2 \abs{S^{n \cdot d -1}} \int_0^t \int_{\R^d} \brac{\pr{x, \tau} \cdot \gr_{\pr{x,\tau}}\pr{ \del_\tau {u} }} u\pr{x,\tau} \pr{2 d \tau }^{\frac{n \cdot d -2}{2}} G_{\tau,n}\pr{x} \d{x} \d{\tau} .
\end{align*}
where we have applied \eqref{vngr} to the second term in the first integral.
Therefore,
\begin{align*}
 L\pr{\sqrt{2 d t}, v_n} 
&= \frac{\sqrt{2 d t} D\pr{\sqrt{2 d t}, v_n}}{H\pr{\sqrt{2 d t}, v_n} } 
= \frac{ \int_{\R^d} u\pr{x,t} \pr{x \cdot \gr u +2t \der{u}{t}} G_{t,n}\pr{x} \d{x}  }{ \int_{\R^d} \abs{u\pr{x,t}}^2 G_{t,n}\pr{x} \d{x}} \\
&- \frac{ 2 \int_0^t \int_{\R^d} \brac{\pr{x, \tau} \cdot \gr_{\pr{x,\tau}}\pr{ \del_\tau {u} }} u\pr{x,\tau} \pr{\frac{ \tau}{t}}^{\frac{n \cdot d -2}{2}} G_{\tau,n}\pr{x}  \d{x} \, \d{\tau}. }{ \int_{\R^d} \abs{u\pr{x,t}}^2 G_{t,n}\pr{x}  \d{x}}.
\end{align*}
Thus,
\begin{align*}
\lim_{n \to \iny} L\pr{\sqrt{2 d t}, v_n} 
&= \frac{ \int_{\R^d} u\pr{x,t} \pr{x \cdot \gr u +2t \der{u}{t}} G_t\pr{x} \d{x} }{ \int_{\R^d} \abs{u\pr{x,t}}^2 G_t\pr{x} \d{x}},
\end{align*}
where we have used Lemma \ref{GnLimit} and that $\disp \lim_{n \to \iny} \pr{\frac{ \tau}{t}}^{\frac{n \cdot d -2}{2}} = 0$ for every $\tau \in \pr{0, t}$ along with the dominated convergence theorem.

Since $\disp \gr G_t\pr{x} = - \frac{x}{2t} G_t\pr{x}$ and $\der{u}{t} = - \LP u$, then
\begin{align*}
\int_{\R^d} u\pr{x,t} \pr{x \cdot \gr u +2t \der{u}{t}} G_t\pr{x} \d{x}
&= - 2 t \int_{\R^d} u\pr{x,t} \gr \cdot \pr{\gr u \, G_t\pr{x}} \d{x} \\
&= 2 t \int_{\R^d} \abs{\gr u\pr{x,t}}^2 \, G_t\pr{x}  \d{x}.
\end{align*}
Therefore,
\begin{align}
\lim_{n \to \iny} L\pr{\sqrt{2 d t}, v_n} 
&= \frac{ 2 t \int_{\R^d} \abs{\gr u\pr{x,t}}^2 \, G_t\pr{x}  \d{x} }{ \int_{\R^d} \abs{u\pr{x,t}}^2 G_t\pr{x} \d{x}}
= 2 \mathcal{L}\pr{t; u}.
\label{LtoL}
\end{align}
By Corollary \ref{eFF2},
\begin{align*}
\der{L\pr{\sqrt{2 d t}; v_n}}{t} 
&\ge \sqrt{\frac{2d}{t}} \frac{\pr{ \int_{S_t^n} v_n\pr{y} \,{ y \cdot \gr v_n\pr{y}} \si_{n \cdot d -1}^t} \pr{ \int_{B_t^n} h_n\pr{y} \, v_n\pr{y} \d{y} }}{ \pr{ \int_{S_t^n} \abs{v_n\pr{y}}^2 \si_{n \cdot d -1}^t}^2 } \\
&- \sqrt{\frac{2d}{t}} \frac{ \int_{B_t^n} h_n\pr{y} { y \cdot \gr v_n\pr{y} } \d{y} }{ \int_{S_t^n} \abs{v_n\pr{y}}^2 \si_{n \cdot d -1}^t}.
\end{align*}
By Lemma \ref{PFT} and \eqref{vngr}
\begin{align*}
& \frac 1 {2 \abs{S^{n \cdot d -1}} } \int_{B_t^n} h_n\pr{y} \pr{ y \cdot \gr v_n\pr{y} }\abs{y}^{n \cdot d -2} \abs{y}^{2 - n \cdot d} \d{y} \\
&=  \int_0^t \int_{\R^d} \brac{ \pr{x, \tau} \cdot \gr_{\pr{x,\tau}}\pr{ \del_\tau {u} }} \pr{x \cdot \gr u +2\tau \del_\tau{u} } \pr{2 d \tau}^{\frac{n \cdot d -2}{2}} G_{n}\pr{x,\tau} \d{x} \d{\tau}.
\end{align*}
Therefore,
\begin{align*}
&\der{L\pr{\sqrt{2 d t}; v_n}}{t} \\
&\ge - \frac{2}{t} \frac{ \int_0^t \int_{\R^d} \brac{ \pr{x, \tau} \cdot \gr_{\pr{x,\tau}}\pr{ \del_\tau {u} }} \pr{x \cdot \gr u +2\tau \del_\tau{u} } \pr{\frac \tau t}^{\frac{n \cdot d -2}{2}} G_{n}\pr{x,\tau} \d{x} \d{\tau} }{ \int_{\R^d} \abs{u\pr{x,t}}^2 G_{t,n}\pr{x} \d{x} } \\
&+ \frac{2}{t} \frac{\pr{  \int_{\R^d} u\pr{x,t} \pr{x \cdot \gr u +2t \der{u}{t}} G_{t, n}\pr{x} \d{x} } }{ \pr{ \int_{\R^d} \abs{u\pr{x,t}}^2 G_{t,n}\pr{x} \d{x} }^2 } \times \\
&\times \pr{ \int_0^t \int_{\R^d} \brac{\pr{x, \tau} \cdot \gr_{\pr{x,\tau}}\pr{ \del_\tau {u} }} u\pr{x,\tau} \pr{\frac \tau t}^{\frac{n \cdot d -2}{2}} G_{\tau,n}\pr{x} \d{x} \d{\tau} }.
\end{align*}
By the same reasoning as above, we conclude that $\disp \lim_{n \to \iny}\der{L\pr{\sqrt{2 d t}; v_n}}{t} \ge 0$ .
It follows from \eqref{LtoL} that $\mathcal{L}$ is monotonically non-decreasing in $t$, as required.
\end{proof}

\section{Free Boundary Problems}
\label{FBP}

In \cite{ACF84}, the authors study two-phase free boundary elliptic problems.
The monotonicity formula presented below is a key tool in their work.
This formula is used to establish Lipschitz continuity of minimizers, to identify blow-up limits, and to prove differentiability of the free boundary when $N=2$.

\begin{thm}[\cite{ACF84}, Lemma 5.1]
Let $v_1, v_2$ be two non-negative functions that belong to $C^0\pr{B_R}\cap H^{1,2}\pr{B_R}$, where $B_R$ is the ball of radius $R$ in $\R^N$.
Assume that $\LP v_1 \ge 0$, $\LP v_2 \ge 0$, $v_1 v_2 \equiv 0$ and $v_1\pr{0} = v_2\pr{0} = 0$.
Then for all $r < R$,
\begin{equation}
\phi\pr{r; v} = \frac{1}{r^4} \pr{ \int_{B_r}  \abs{\gr v_1\pr{y}}^2 \abs{y}^{2-N} \d{y} } \pr{ \int_{B_r} \abs{\gr v_2\pr{y}}^2 \abs{y}^{2-N} \d{y}}
\label{2PPhiDef}
\end{equation}
is monotonically non-decreasing in $r$.
\label{HFBP}
\end{thm}

In the proof of the parabolic version of this theorem, given below, we employ the following non-homogeneous version of this result.

\begin{cor}
Let $v_1, v_2$ be two non-negative functions that belong to $C^0\pr{B_R}\cap H^{1,2}\pr{B_R}$.
Assume that $\LP v_1 \ge h_1$, $\LP v_2 \ge h_2$, $v_1 v_2 \equiv 0$ and $v_1\pr{0} = v_2\pr{0} = 0$.
Assume further that for every $r < R$, $\Ga_{1, r} := \supp v_1 \cap \del B_r$ and $\Ga_{2, r} := \supp v_2 \cap \del B_r$ have non-zero measure.
Then for all $r < R$, if we define $\phi\pr{r; v}$ as in \eqref{2PPhiDef}, then
\begin{align*}
\phi^\prime\pr{r; v}
&\ge \frac{2 }{r^4} \psi\pr{s_{1,r}}  \pr{ \int_{B_r}  v_1 h_1 \abs{y}^{2-N} \d{y} } \pr{ \int_{B_r} \abs{\gr v_2 }^2\abs{y}^{2-N} \d{y}}  \\
&+ \frac{2 }{r^4} \psi\pr{s_{2,r}}  \pr{ \int_{B_r} \abs{\gr v_1}^2 \abs{y}^{2-N} \d{y} } \pr{ \int_{B_r}  v_2  h_2 \abs{y}^{2-N}  \d{y} },
\end{align*}
where $\psi$ is given below in \eqref{psiFuncDef} and $\disp s_{i,r} := \frac{\abs{\Ga_{i, r}}}{r^{N-1} \abs{S^{N-1}}}$ for $i = 1,2$.
\label{HFBPCor}
\end{cor}

The proof of this corollary follows that of the proof of Theorem \ref{HFBP}, but we include it here for completeness.

\begin{proof}
Let $v_{1,m} = \rho_{1/m} * v_1$, where $\rho_\eps$ denotes the standard mollifier.
Set $h_{1, m} = \rho_{1/ m} * h_1$.
Then $\LP v_{1, m} \ge h_{1, m}$ and therefore $\LP \pr{v_{1, m}^2} \ge 2 \abs{\gr v_{1, m}}^2 + 2 v_{1, m} h_{1, m}$.
Hence, for $0 < \eps << r$,
\begin{align*}
& 2 \int_{B_r \setminus B_\eps} \abs{\gr v_{1,m}}^2 \abs{y}^{2-N}  \d{y} \\
&\le \int_{B_r \setminus B_\eps} \LP\pr{ v_{1,m}^2} \abs{y}^{2-N}  \d{y}
- 2 \int_{B_r \setminus B_\eps}  v_{1,m} h_{1,m} \abs{y}^{2-N}  \d{y} \\
&=  2 r^{1- N}  \int_{\del B_r } v_{1,m} \pr{y \cdot \gr v_{1,m}} \d{S}\pr{y}
+  \pr{N - 2} r^{1-N}  \int_{ \del B_r } \abs{ v_{1,m}}^2 \d{S}\pr{y} \\
&- 2 \int_{B_r \setminus B_\eps}  v_{1,m} h_{1,m} \abs{y}^{2-N}  \d{y}
- I_\eps ,
\end{align*}
where
$$I_\eps = 2 \eps^{1 -N } \int_{\del B_\eps } v_{1,m} \pr{y \cdot \gr v_{1,m}} \d{S}\pr{y}
+ \pr{N - 2} \eps^{ 1 - N} \int_{ \del B_\eps } \abs{ v_{1,m}}^2  \d{S}\pr{y} .$$
Since $\gr v_{1,m}$ is bounded, then $I_\eps \to \pr{N-2} \abs{S^{N-1}} v_{1, m}\pr{0}^2$ as $\eps \to 0$.
Therefore, integrating from $r_0$ to $r_0 + \de$, dividing through by $\de$ and taking the limit as $m \to \iny$, we see that
\begin{align*}
& \frac{2}{\de} \int_{r_0}^{r_0 + \de} dr \int_{B_r \setminus B_\eps} \abs{\gr v_1}^2 \abs{y}^{2-N}  \d{y} \\
&\le \frac{2}{\de} \int_{r_0}^{r_0 + \de} r^{1- N}  dr \int_{\del B_r } v_1 \pr{ y \cdot \gr v_1} \d{S}\pr{y} \\
&+ \frac{N- 2}{\de} \int_{r_0}^{r_0 + \de} r^{1-N} dr \int_{ \del B_r } \abs{ v_1}^2 \d{S}\pr{y} 
- \frac{2}{\de} \int_{r_0}^{r_0 + \de} dr \int_{B_r \setminus B_\eps}  v_1 h_1 \abs{y}^{2-N}  \d{y}.
\end{align*}
Letting $\de \to 0$, it follows that for a.e. $r_0$,
\begin{align*}
&2 \int_{B_{r_0} \setminus B_\eps} \abs{\gr v_1}^2 \abs{y}^{2-N}  \d{y}  \\
&\le 2 r_0^{1- N} \int_{\del B_{r_0} } v_1 \, y \cdot \gr v_1 \d{S}\pr{y}
+ \pr{N- 2} r_0^{1-N} \int_{ \del B_{r_0} } \abs{ v_1}^2 \d{S}\pr{y} \\
&- 2 \int_{B_{r_0} \setminus B_\eps}  v_1 h_1 \abs{y}^{2-N}  \d{y}.
\end{align*}
Therefore, for a.e. $r$, we have
\begin{align}
&\int_{B_r} \abs{\gr v_1}^2 \abs{y}^{2-N}  \d{y}
\le r^{1- N} \int_{\del B_r } v_1 \, y \cdot \gr v_1 \d{S}\pr{y}  \nonumber \\
&+ \frac{N- 2}{2} r^{1-N} \int_{ \del B_r } \abs{ v_1}^2 \d{S}\pr{y} 
- \int_{B_r}  v_1 h_1 \abs{y}^{2-N}  \d{y}.
\label{grv+Est}
\end{align}
Moreover, for a.e. $r$,
\begin{align*}
\frac{d}{dr} \int_{B_r} \abs{\gr v_1}^2 \abs{y}^{2 - N} \d{y} 
= r^{2 - N} \int_{\del B_r} \abs{\gr v_1}^2  \d{S}\pr{y}.
\end{align*}
Analogous statements may be made with $v_2$ in place of $v_1$ and therefore, for a.e. r,
\begin{align}
&\phi^\prime\pr{r; v} 
=- \frac{4}{r^5} \pr{ \int_{B_r}  \abs{\gr v_1}^2 \abs{y}^{2-N} \d{y} } \pr{ \int_{B_r} \abs{\gr v_2}^2 \abs{y}^{2-N} \d{y}}  
\nonumber \\
&+ \frac{r^{2-N}}{r^4} \pr{  \int_{\del B_r}  \abs{\gr v_1}^2 \d{S}\pr{y} } \pr{ \int_{B_r} \abs{\gr v_2}^2 \abs{y}^{2-N} \d{y}} \nonumber \\
&+ \frac{r^{2-N}}{r^4} \pr{ \int_{B_r}  \abs{\gr v_1}^2 \abs{y}^{2-N} \d{y} } \pr{ \int_{\del B_r} \abs{\gr v_2}^2 \d{S}\pr{y} }.
\label{phiPrime}
\end{align}

We now want to estimate this derivative above.
Assume first that $r =1$.
Let $\gr_\te w$ denote the gradient of  function $w$ on $S^{N-1}$, the unit sphere.
Let $\Ga_i$ denote the support of $v_i$ on $S^{N-1}$ for $i = 1, 2$. 
By assumption, the measures of ${\Ga_1}$ and ${\Ga_2}$ are non-zero.

For $i = 1,2 $, define
\begin{align*}
\frac{1}{\al_i} = \inf_{w \in H^{1,2}_0\pr{\Ga_i}} \frac{\int_{\Ga_i} \abs{\gr_\te w}^2 }{\int_{\Ga_{i}} w^2}.
\end{align*}
Then for any $\be_1 \in \pr{0,1}$,
\begin{align*}
&\frac{2 \be_1}{\sqrt{\al_1}} \int_{\del B_1} \abs{v_1} \abs{ y \cdot \gr v_1} \d{S}\pr{y} \\
&\le 2 \pr{ \frac{\be_1^2}{\al_1} \int_{\del B_1} \abs{v_1}^2 \d{S}\pr{y}}^{\frac 1 2} \pr{  \int_{\del B_1} \pr{y \cdot \gr v_1}^2 \d{S}\pr{y} }^{\frac 1 2} \\
&\le 2 \pr{ \be_1^2  \int_{\del B_1} \abs{\gr_\te v_1}^2 \d{S}\pr{y}}^{\frac 1 2} \pr{ \int_{\del B_1} \pr{ \gr_r v_1}^2 \d{S}\pr{y} }^{\frac 1 2} \\
&\le \int_{\del B_1} \brac{ \be_1^2  \abs{\gr_\te v_1}^2 +\pr{ \gr_r v_1}^2 }\d{S}\pr{y} 
\end{align*}
and
\begin{align*}
\frac{1 - \be_1^2}{\al_1} \int_{\del B_1} \abs{v_1}^2 \d{S}\pr{y} 
\le \pr{1 - \be_1^2} \int_{\del B_1} \abs{\gr_\te v_1}^2 \d{S}\pr{y} .
\end{align*}
If we set 
\begin{equation}
\frac{1 - \be_{i}^2}{\al_i} = \pr{N-2} \frac{\be_i}{\sqrt{\al_i}}
\label{albeRel}
\end{equation}
for $i = 1,2$, then by combining \eqref{grv+Est} with the last two inequalities, we have
\begin{align*}
&\frac{2\be_1}{\sqrt{\al_1}} \int_{B_1} \abs{\gr v_1}^2 \abs{y}^{2-N}  \d{y}
+ \frac{2\be_1}{\sqrt{\al_1}} \int_{B_1}  v_1 h_1 \abs{y}^{2-N}  \d{y} \\
&\le \frac{2\be_1}{\sqrt{\al_1}} \int_{\del B_1} \abs{v_1 } \abs{ y \cdot \gr v_1} \d{S}\pr{y} 
+ \pr{N-2} \frac{\be_i}{\sqrt{\al_i}} \int_{\del B_1} \abs{v_1}^2 \d{S}\pr{y} \\
&\le \int_{\del B_1} \abs{\gr v_1}^2 \d{S}\pr{y}  .
\end{align*}
The same bound holds with $u_2$, $\al_2$, and $\be_2$ in place of $u_1$, $\al_1$ and $\be_1$, respectively.
Substituting these inequalities into \eqref{phiPrime} gives
\begin{align*}
\phi^\prime\pr{1; v} 
&\ge 2\pr{\frac{\be_1}{\sqrt{\al_1}} + \frac{\be_2}{\sqrt{\al_2}} - 2} \pr{ \int_{B_1}  \abs{\gr v_1}^2 \abs{y}^{2-N} \d{y} } \pr{ \int_{B_1} \abs{\gr v_2}^2 \abs{y}^{2-N} \d{y}}  \\
&+2 \frac{\be_1}{\sqrt{\al_1}} \pr{ \int_{B_1}  v_1 \, h_1 \abs{y}^{2-N} \d{y} } \pr{ \int_{B_1} \abs{\gr v_2}^2 \abs{y}^{2-N} \d{y}} \\
&+ 2\frac{\be_2}{\sqrt{\al_2}} \pr{ \int_{B_1}  \abs{\gr v_1}^2 \abs{y}^{2-N} \d{y} } \pr{ \int_{B_1}  v_2 \,h_2 \abs{y}^{2-N}  \d{y} }.
\end{align*}
The relation \eqref{albeRel} is satisfied when
$$\frac{\be_i}{\sqrt{\al_i}} = \frac 1 2 \set{\brac{\pr{N-2}^2 + \frac{4}{\al_i}}^{\frac 1 2} - \pr{N-2}}.$$
If we define $\ga_i > 0$ so that $\disp \ga_i \pr{\ga_i + N - 2} = \frac 1 {\al_i}$ then $\disp \frac{\be_i}{\sqrt{\al_i}} = \ga_i$ for $i = 1,2$.

As a function that acts on subsets of $S^{N-1}$, $\ga$ was studied in \cite{FH76} and it was shown that $\disp \ga\pr{E} \ge \psi\pr{\frac{\abs{E}}{\abs{S^{N-1}}}}$, where $\psi$ is the decreasing, convex function defined by
\begin{equation}
\psi\pr{s} = \left\{\begin{array}{ll} \frac 1 2 \log\pr{\frac 1 {4s}} + \frac 3 2 & \text{ if } s < \frac 1 4 \\
2 \pr{1 - s} & \text{ if } \frac 1 4 < s < 1.  \end{array}  \right.
\label{psiFuncDef}
\end{equation}
We use the notation $\ga_i = \ga\pr{\Ga_i}$ for $i = 1,2$. 
With $\disp s_i = \frac{\abs{\Ga_i}}{\abs{S^{N-1}}}$, it follows from convexity that
\begin{align*}
\ga_1 + \ga_2
\ge \psi\pr{s_1} + \psi\pr{s_2}
\ge 2\psi\pr{\frac{s_1 + s_2}{2}}
\ge 2\psi\pr{\frac{1}{2}}
= 2.
\end{align*}
Therefore,
\begin{align*}
\phi^\prime\pr{1; v} 
&\ge 2 \psi\pr{s_1} \pr{ \int_{B_1}  v_1 \, h_1 \abs{y}^{2-N} \d{y} } \pr{ \int_{B_1} \abs{\gr v_2}^2 \abs{y}^{2-N} \d{y}} \\
&+ 2 \psi\pr{s_2} \pr{ \int_{B_1}  \abs{\gr v_1}^2 \abs{y}^{2-N} \d{y} } \pr{ \int_{B_1}  v_2 \, h_2 \abs{y}^{2-N}  \d{y} }.
\end{align*}

For values of $r \ne 1$, define $v_{i, r}\pr{y} = r^{-1} v_i\pr{r y}$ for $i = 1, 2$, 
so that
\begin{align*}
\phi\pr{1; v_r}
&=  \pr{ \int_{B_1}  \abs{\gr v_{1,r}\pr{y}}^2 \abs{y}^{2-N} \d{y} } \pr{ \int_{B_1} \abs{\gr v_{2,r}\pr{y}}^2 \abs{y}^{2-N} \d{y}} \\
&= \frac{1}{r^4} \pr{ \int_{B_r}  \abs{\gr v_1}^2 \abs{y}^{2-N} \d{y} } \pr{ \int_{B_r} \abs{\gr v_2}^2 \abs{y}^{2-N} \d{y} }
= \phi\pr{r, v}.
\end{align*}
With $h_{i, r}\pr{y} = r h_i\pr{r y}$
, we have $\LP v_{i, r} \ge h_{i, r}$.
Let $\Ga_{i, r}$ denote the support of  $v_{i, r}$ on $\del B_r$ and set $\disp s_{i, r} = \frac{\abs{\Ga_{i, r}}}{r^{N-1} \abs{S^{N-1}}}$ for $i = 1, 2$.
Applying the derivative estimates above to the pair $v_{1, r}$, $v_{2, r}$ then rescaling leads to the conclusion.
\end{proof}

Motivated by its application to the regularity theory of two-phase free boundary elliptic problems, the parabolic analogue of the monotonicity formula due to Alt-Cafferelli-Friedman was proved by Cafarelli in \cite{Caf93}.
This two-phase monotonicity formula was extended by Cafferelli and Kenig in \cite{CK98} and used to prove uniform Lipschitz estimates for solutions to singular perturbations of variable coefficient parabolic free boundary problems, where the linear parabolic operators are second-order divergence form with Dini top order coefficients.

\begin{thm}[\cite{Caf93}, Theorem 1]
Let $u_1, u_2$ be two non-negative functions that belong to $C^0\pr{\R^d \times \pr{0,T}}\cap H^{1,2}\pr{\R^d \times \pr{0,T}}$.
Assume that $\LP u_1 + \del_t u_1 \ge 0$, $\LP u_2 + \del_t u_2 \ge 0$, $u_1 u_2 \equiv 0$ and $u_1\pr{0,0} = u_2\pr{0,0} = 0$.
Assume also that $u_1$ and $u_2$ have moderate growth at infinity.
Let $G_t\pr{x}$ be as given in \eqref{GDefn}.
Then for all $\tau < T$,
\begin{equation*}
\Phi\pr{\tau; u} = \frac{1}{\tau^2} \pr{\int_0^\tau \int_{\R^d}  \abs{\gr u_1}^2 G_t\pr{x} \d{x} \, \d{t}} \pr{\int_0^\tau \int_{\R^d}  \abs{\gr u_2}^2 G_t\pr{x} \d{x} \, \d{t}},
\end{equation*}
is monotonically non-decreasing in $\tau$.
\label{PFBP}
\end{thm}

We reprove this theorem using only Corollary \ref{HFBPCor} and the tools developed in the early sections of this article.

\begin{proof}
Let $u_1$, $u_2$ be as in the statement of the theorem.
For each $n$, define $v_{1, n}, v_{2, n} : B_{T}^n \su \R^{n \cdot d} \to \R$ so that
$$v_{1, n}\pr{y} = u_1\pr{F_{n,d}\pr{y}}, \quad v_{2, n}\pr{y} = u_2\pr{F_{n,d}\pr{y}} .$$
For every $n$, the pair $v_{1, n}$, $v_{2, n}$ is continuous and non-negative.
Moreover, $v_{1,n}\pr{0} = u_1\pr{0,0} = 0$, $v_{2,n}\pr{0} = u_2\pr{0,0} = 0$ and $v_{1,n} v_{2,n} \equiv 0$.
Then, by \eqref{Laplace}
\begin{align*}
\LP v_{1, n} 
&= n \pr{\LP u_1 + \der{u_1}{t}} + \frac{2}{d} \pr{x, t} \cdot \gr_{\pr{x,t}} \pr{\del_t{u_1} }
\ge \frac{2}{d} \pr{x, t} \cdot \gr_{\pr{x,t}} \pr{\del_t{u_1} } \\
&=: J_1\pr{x,t} \\
\LP v_{2, n} 
&= n \pr{\LP u_2 + \der{u_2}{t}} + \frac{2}{d} \pr{x, t} \cdot \gr_{\pr{x,t}} \pr{\del_t{u_2} }
\ge \frac{2}{d} \pr{x, t} \cdot \gr_{\pr{x,t}} \pr{\del_t{u_2} } \\
&=: J_2\pr{x,t}.
\end{align*}
Define $h_{1, n}, h_{2, n} : B_{T}^n \to \R$ so that 
$$h_{1, n}\pr{y} = J_1\pr{F_{n,d}\pr{y}}, \quad h_{2, n}\pr{y} = J_2\pr{F_{n,d}\pr{y}}.$$
Then $\LP v_{1, n} \ge h_{1, n}$ and $\LP v_{2, n} \ge h_{2, n}$ so we may apply Corollary \ref{HFBPCor}.

Let $\Ga_{1, n, \tau} = \supp v_{1, n} \cap S_\tau^n$ and let $\Ga_{2, n, \tau} = \supp v_{2, n} \cap S_\tau^n$.
For each $t$, the measure of $\supp u_1\pr{\cdot, t}$ vanishes if and only if the measure of $\Ga_{1, n, t}$ vanishes for every $n$.
Similarly, the measure of $\supp u_2\pr{\cdot, t}$ vanishes if and only if the measure of $\Ga_{2, n, t}$ vanishes for every $n$.
We'll assume first that for every $t$, the measures of $\supp u_1\pr{\cdot, t}$ and $\supp u_2\pr{\cdot, t}$ are non-vanishing.
Therefore, for every $i$, $n$ and $t$, $\Ga_{i, n, t}$ has non-zero measure. 
Thus, we may apply Corollary \ref{HFBPCor} to each pair $v_{1, n}$, $v_{2, n}$.

Define $\disp \Phi_n\pr{\tau} = \frac{4 }{\pr{n \abs{S^{n \cdot d - 1}}}^2} \phi\pr{\sqrt{2 d \tau}; v_n}$.
By Lemma \ref{PFT}
\begin{align*}
& \int_{B_\tau^n} \abs{\gr v_{1,n} }^2\abs{y}^{2-n \cdot d} \d{y} \\
&= nd \abs{S^{n \cdot d - 1 }} \int_0^\tau \int_{\R^d} \brac{ \abs{\gr u_1}^2 + \frac{2}{n d} \brac{ {\pr{x,t} \cdot \gr_{\pr{x,t}} u_1} }\del_t{u_1}} G_n\pr{x,t} \d{x} \, \d{t}
\end{align*}
and a similar equality holds for $u_2$.
Therefore,
\begin{align*}
\Phi_n\pr{\tau}
&=  \frac{4 }{\pr{n \abs{S^{n \cdot d - 1}}}^2} \frac{1}{\pr{2 d \tau}^2} \pr{ \int_{B_\tau^n} \abs{\gr v_{1, n}\pr{y}} \abs{y}^{2 - n \cdot d} \d{y} }\times \\
&\times  \pr{ \int_{B_\tau^n} \abs{\gr v_{2, n}\pr{y}} \abs{y}^{2 - n \cdot d} \d{y} } \\
&= \frac{1}{\tau^2} \int_0^\tau \int_{\R^d} \pr{ \abs{\gr u_1}^2 + \frac 2 {nd} \brac{\pr{x, t} \cdot \gr_{\pr{x,t}} u_1} \del_t{u_1}  } G_n\pr{x,t}\d{x} \, \d{t}  \\
&\times \int_0^\tau \int_{\R^d}  \pr{ \abs{\gr u_2}^2 + \frac 2 {nd} \brac{\pr{x, t} \cdot \gr_{\pr{x,t}} u_2} \del_t{u_2}  } G_n\pr{x,t} \d{x} \, \d{t} .
\end{align*}
It follows from Lemma \ref{GnLimit} that
\begin{align}
&\lim_{n \to \iny} \Phi_n\pr{\tau} \nonumber \\
&= \frac{1}{\tau^2} \pr{\int_0^\tau \int_{\R^d} \abs{\gr u_1}^2 G_t\pr{x} \d{x} \, \d{t} } \pr{\int_0^\tau \int_{\R^d} \abs{\gr u_2}^2 G_t\pr{x} \d{x} \, \d{t} } \nonumber \\
&= \Phi\pr{\tau; u}.
\label{PhinLimit}
\end{align}

By Corollary \ref{HFBPCor}
\begin{align*}
\Phi_n^\prime\pr{\tau}
&= \frac{4 }{\pr{n \abs{S^{n \cdot d - 1}}}^2} \phi^\prime\pr{\sqrt{2 d \tau}; v_n} \sqrt{\frac{d}{2 \tau}} \\
&\ge \sqrt{\frac{2d}{\tau}} \frac{1 }{\pr{n d \tau \abs{S^{n \cdot d - 1}}}^2} \psi\pr{s_{1,n,\tau}} \pr{ \int_{B_\tau^n}  v_{1,n} h_{1,n} \abs{y}^{2-n \cdot d} \d{y} } \times \\
&\times \pr{ \int_{B_\tau^n} \abs{\gr v_{2,n} }^2\abs{y}^{2-n \cdot d} \d{y}}  \\
&+  \sqrt{\frac{2d}{\tau}} \frac{1 }{\pr{n d \tau \abs{S^{n \cdot d - 1}}}^2}  \psi\pr{s_{2,n,\tau}} \pr{ \int_{B_\tau^n} \abs{\gr v_{1,n}}^2 \abs{y}^{2-n \cdot d} \d{y} } \times \\
&\times\pr{ \int_{B_\tau^n}  v_{2,n}  h_{2,n} \abs{y}^{2-n \cdot d}  \d{y} },
\end{align*}
where we have set $\disp s_{i, n, \tau} = \pr{2 d \tau}^{-\frac{n \cdot d -1}{2}} \frac{\abs{\Ga_{i, n, \tau}}}{\abs{S^{n \cdot d -1}}}$.  
Lemma \ref{PFT} implies that
\begin{align*}
&\int_{B_\tau^n}  v_{1,n}  h_{1,n} \abs{y}^{2-n \cdot d}  \d{y} \\
&= 2 \abs{S^{n \cdot d - 1}} \int_0^\tau \int_{\R^d} u_1 \brac{\pr{x, t} \cdot \gr_{\pr{x,t}} \pr{\del_t{u_1} }} G_n\pr{x,t} \d{x} \, \d{t}
\end{align*}
and an analogous equality for $u_2$.
Therefore,
\begin{align}
&\Phi_n^\prime\pr{\tau}
\ge \sqrt{\frac{8}{d \tau^5}} \frac{\psi\pr{s_{1,n,\tau}}}{n} \pr{\int_0^\tau \int_{\R^d} u_1 \brac{\pr{x, t} \cdot \gr_{\pr{x,t}} \pr{\del_t{u_1} }} G_n\pr{x,t} \d{x} \, \d{t}} \times \nonumber \\
&\times \pr{\int_0^\tau \int_{\R^d} \brac{ \abs{\gr u_2}^2 + \frac{2}{n d} \brac{ {\pr{x,t} \cdot \gr_{\pr{x,t}} u_2} }\del_t{u_2}} G_n\pr{x,t} \d{x} \, \d{t} } \nonumber \\
&+   \sqrt{\frac{8}{d \tau^5}} \frac{\psi\pr{s_{2,n,\tau}}}{n} \pr{ \int_0^\tau \int_{\R^d} u_2 \brac{\pr{x, t} \cdot \gr_{\pr{x,t}} \pr{\del_t{u_2} }} G_n\pr{x,t} \d{x} \, \d{t} } \times \nonumber \\
&\times  \pr{ \int_0^\tau \int_{\R^d} \brac{ \abs{\gr u_1}^2 + \frac{2}{n d} \brac{ {\pr{x,t} \cdot \gr_{\pr{x,t}} u_1} }\del_t{u_1}} G_n\pr{x,t} \d{x} \, \d{t} }
\label{PhinPrime}
\end{align}

To proceed, we need to examine $\psi\pr{s_{i,n,t}}$.
Let $\chi_1$ denote the indicator function of the support of $u_1$.
Define $\mu_{1, n}$ so that $\mu_{1, n}\pr{y} = \chi_1\pr{F_{n,d}\pr{y}}$.
Then $\mu_{1, n}$ is the indicator function of the support of $v_{1, n}$.
Consequently, by Lemma \ref{PFTS}
\begin{align*}
\abs{\Ga_{1, n, t}} 
&= \abs{\supp v_{1, n} \cap S_t^n} 
= \int_{S_t^n} \mu_{1, n}\pr{y} \si_{n \cdot d -1}^t
= \int_{S_t^n} \chi_1\pr{f_{n,d}\pr{y}, t} \si_{n \cdot d -1}^t\\
&=  \pr{2 d t}^{\frac{n \cdot d - 1}{2}} \abs{S^{n \cdot d - 1}} \int_{\R^d} \chi_1\pr{x, t}G_{t,n}\pr{x} \d{x}
\end{align*}
so that
\begin{align*}
s_{1, n, t} 
&= \int_{\R^d} \chi_1\pr{x, t}G_{t,n}\pr{x} \d{x}. 
\end{align*}
By Lemma \ref{GnLimit}
\begin{align*}
\lim_{n \to \iny} s_{1, n, t} 
&=  \int_{\R^d} \chi_1\pr{x, t} G_t\pr{x} \d{x}. 
\end{align*}
Since we have assumed that for every $t$, $\chi_1\pr{\cdot, t}$ is always non-trivial, then the integral above is non-zero and for every $t$ and for $n$ sufficiently large, $s_{1, n, t}$ is bounded away from zero. 
Therefore, from \eqref{psiFuncDef}, for $n$ sufficiently large, each $\psi\pr{s_{1,n,t}}$ is bounded above.
By the same reasoning, each $\psi\pr{s_{2,n,t}}$ is also bounded from above for $n$ sufficiently large.
It follows from \eqref{PhinPrime} that
\begin{align*}
\lim_{n \to \iny} \Phi_n^\prime\pr{\tau} \ge 0.
\end{align*}
By \eqref{PhinLimit}, the proof is complete under the assumption that the measures of $\supp u_1\pr{\cdot, t}$ and $\supp u_2\pr{\cdot, t}$ are non-vanishing for every $t$.

Now assume that there exists some values of $t$ such that the measure of $\supp u_1\pr{\cdot, t}$ or the measure of $\supp u_2\pr{\cdot, t}$ vanishes.
Let $\tau$ be the largest such $t$-value.
Assume, without loss of generality, that $\abs{\supp u_1\pr{\cdot, \tau}} = 0$.
Then, for every $n$, $\abs{\Ga_{1, n, \tau}} = 0$ as well.
Therefore, for every $n$, $\LP v_{1,n} \ge h_{1, n}$ in $D_{1, n, \tau} := \supp v_{1,n} \cap B_\tau^n$ with $v_{1, n} = 0$ along $\del D_{1, n, \tau}$.
By estimate \eqref{grv+Est} applied to $v_{1,n}$ on $D_{1,n,\tau}$,
\begin{align*}
\int_{D_{1, n \tau}} \abs{\gr v_{1,n}}^2 \abs{y}^{2 - n \cdot d} \d{y} \le - \int_{D_{1, n, \tau}} v_{1,n} h_{1,n} \abs{y}^{2 - n \cdot d} \d{y}
\end{align*}
so that
\begin{align*}
\int_{B_\tau^n} \abs{\gr v_{1,n}}^2 \abs{y}^{2 - n \cdot d} \d{y} 
\le - \int_{B_\tau^n} v_{1,n} h_{1,n} \abs{y}^{2 - n \cdot d} \d{y}.
\end{align*}
It follows from the computations above (based on applications of Lemma \ref{PFT}) that for every $n$,
\begin{align*}
& \int_0^\tau \int_{\R^d} \abs{\gr u_1}^2  G_n\pr{x,t} \d{x} \, \d{t} \\
&\le -\frac{2}{nd} \int_0^\tau \int_{\R^d} \brac{\pr{x, t} \cdot \gr_{\pr{x,t}} \del_t\pr{u_1}^2 } G_n\pr{x,t} \d{x} \, \d{t}
\end{align*}
By Lemma \ref{GnLimit}
\begin{align*}
&\int_0^\tau \int_{\R^d} \abs{\gr u_1}^2 G\pr{x, t} \d{x} \, \d{t}  \\
&= \lim_{n \to \iny}  \int_0^\tau \int_{\R^d} \abs{\gr u_1}^2  G_n\pr{x,t} \d{x} \, \d{t} \\
&\le - \lim_{n \to \iny} \frac{2}{nd} \int_0^\tau \int_{\R^d} \brac{\pr{x, t} \cdot \gr_{\pr{x,t}} \del_t\pr{u_1}^2 } G_n\pr{x,t} \d{x} \, \d{t}
= 0.
\end{align*}
Then $\Phi\pr{t, u} = 0$ for every $t \le \tau$ so that $\Phi\pr{t,u}$ is monotonically non-decreasing on that time interval.
Since $\abs{\Ga_{1, n, t}} \ne 0$ and $\abs{\Ga_{2, n, t}} \ne 0$ for every $n$ and every $t > \tau$, then by the arguments from the first case, $\Phi\pr{t,u}$ is monotonically non-decreasing whenever $t > \tau$, completing the proof.
\end{proof}

\section{Harmonic Maps into Spheres}
\label{HM}

In this section, we use a monotonicity result for harmonic maps to derive the proof of the parabolic analogue.  
To start, we introduce some notation as it appears in the introduction to \cite{E91}.

Let $m, N \ge 2$.  
Let $U \subset \R^N$ be smooth, and $S^{m-1}$ denote the unit sphere in $\R^m$.
A function $\vec{v} = \pr{v^1, \ldots, v^m}$ in the Sobolev space $H^1\pr{U; \R^{m}}$ belongs to the space $H^1\pr{U; S^{m-1}}$ if $\abs{\vec{v}} = 1$ almost everywhere in $U$.

\begin{defn}
A function $\vec{v} \in H^1\pr{U; S^{m-1}}$ is a weakly harmonic mapping of $U$ into $S^{m-1}$ provided
\begin{equation}
- \LP \vec v = \abs{D\vec v}^2 \vec v
\label{HMEq}
\end{equation} 
holds weakly in $U$.  That is, for every test function $\vec w = \pr{w^1, \ldots, w^m} \in H^1_0\pr{U; \R^m} \cap L^\iny\pr{U; \R^m}$, we have
\begin{equation}
\int_U D\vec v : D\vec w \; \d{y} = \int_U \abs{D\vec v}^2 \vec v \cdot \vec w \; \d{y},
\label{weakHM}
\end{equation}
where we use the notation
$$D\vec v = \pr{\pr{\frac{\del v^i}{\del y_k}}}_{1 \le i \le m, 1 \le k \le N}$$
$$D\vec v : D\vec w = \sum_{i=1}^m \sum_{k=1}^N \frac{\del v^i}{\del y_k} \frac{\del w^i}{\del y_k}, \;\;\;\;\;\;\;
\abs{D\vec v}^2 = D\vec v : D\vec v.$$
\end{defn}

Now let $\vec g: \del U \to S^{m-1}$ be a given smooth function and set 
$$\mathcal{A} = \set{w \in H^1\pr{U; S^{m-1}}: \vec w = \vec g \text{ on $\del U$ in the trace sense}}.$$
Then \eqref{HMEq} is the Euler-Lagrange equation for the variational problem of minimizing the Dirichlet energy
$$I\brac{\vec w} = \int_U \abs{D\vec w}^2 \d{y}$$
over all $\vec w \in \mathcal{A}$.  
If $\vec v$ is a minimizer of $I\brac{\cdot}$ over $\mathcal{A}$, then $\vec v$ satisfies \eqref{weakHM} and for every vector field $\boldsymbol{ \zeta} = \pr{\zeta^1, \ldots, \zeta^d} \in C^1_0\pr{U; \R^N}$ there holds
\begin{equation}
\int_U \brac{\abs{D\vec v}^2 \pr{\dv \zeta} - 2 \frac{\del v^i}{\del y_k} \frac{\del v^i}{\del y_l} \frac{\del \zeta^k}{\del y_l}} \d{y} = 0.
\label{vectorId}
\end{equation} 

\begin{defn}
A function $\vec v \in H^1\pr{U; S^{m-1}}$ is said to be a weakly stationary harmonic map from $U$ into the sphere $S^{m-1}$ if $\vec v$ satisfies \eqref{weakHM} and \eqref{vectorId} for all test functions $\vec w$ and $\zeta$ as above.
\end{defn}

One way to understand this definition is that \eqref{weakHM} states that $\vec v$ is stationary with respect to the variations of the target $S^{m-1}$, while \eqref{vectorId} states that $\vec v$ is stationary with respect to variations of the domain $U$.  Note that if $\vec v$ is smooth, then \eqref{vectorId} is an immediate consequence of \eqref{weakHM} by taking $\vec w = D\vec v \cdot \zeta$ .

The following is the monotonicity property for weakly stationary harmonic maps.
For generalizations and important applications to regularity theory, see \cite{SU82}, \cite{S84}.
The presentation here is from \cite{E91}.

\begin{thm}[\cite{E91}]
Suppose $y_0 \in U \subset \R^N$ and $R > 0$ is such that $B\pr{y_0, R} \subset U$. For all $r \in \pr{0,R}$, if $\vec v$ is a weakly stationary harmonic map from $U$ into $S^{m-1}$, then the quantity
\begin{equation}
\phi\pr{r; \vec v} = \frac{1}{r^{N-2}} \int_{B\pr{y_0,r}} \abs{D\vec v\pr{y}}^2 \d{y}
\label{phiHMDef}
\end{equation}
is monotonically non-decreasing in $r$.
\label{HMMono}
\end{thm}

To prove the parabolic theorem below, we apply a non-homogeneous version of this result.
We begin with the analogous definitions.

\begin{defn}
A function $\vec{v} \in H^1\pr{U; S^{m-1}}$ is a weak solution to
\begin{equation}
- \LP \vec v = \abs{D\vec v}^2 \vec v + \vec H + h \, \vec v
\label{HMEqiH}
\end{equation} 
if for every test function $\vec w = \pr{w^1, \ldots, w^m} \in H^1_0\pr{U; \R^m} \cap L^\iny\pr{U; \R^m}$, we have
\begin{equation}
\int_U D\vec v : D\vec w \; \d{y} 
= \int_U \pr{\abs{D\vec v}^2 \vec v \cdot \vec w + \vec H \cdot \vec w+ h \, \vec v \cdot \vec w } \d{y}.
\label{weakHMiH}
\end{equation}
\end{defn}

If $\vec v$ is sufficiently smooth, then by taking $\vec w = D\vec v \cdot \zeta$ in \eqref{HMEqiH} we have
\begin{equation}
\int_U \brac{\abs{D\vec v}^2 \pr{\dv \zeta} - 2 \frac{\del v^i}{\del y_k} \frac{\del v^i}{\del y_l} \frac{\del \zeta^k}{\del y_l}} \d{y} 
= \int_U H^i \der{v^i}{y_k} \zeta^k
\label{vectorIdiH}
\end{equation} 
for every vector field $\boldsymbol{ \zeta} = \pr{\zeta^1, \ldots, \zeta^d} \in C^1_0\pr{U; \R^N}$.

Using a variation of the proof presented in \cite{E91}, we show the following version of the monotonicity formula.

\begin{cor}
Suppose $y_0 \in U \subset \R^N$ and $R > 0$ is such that $B\pr{y_0, R} \subset U$. 
Assume that $\vec v: U \to S^{m-1}$ satisfies \eqref{HMEqiH} and \eqref{vectorIdiH}.
Define $\phi\pr{r; \vec v}$ as in \eqref{phiHMDef}.
Then for almost every $r \in \pr{0, R}$,
\begin{align*}
\phi^\prime\pr{r; \vec v} 
&\ge - \frac{1}{r^{N-1}} \int_{B\pr{y_0, r}} \vec H \cdot \brac{\pr{y - y_0} \cdot D \vec v} \, \d{y}.
\end{align*}
\label{HMMonoiH}
\end{cor}

\begin{proof}
Assume without loss of generality that $y_0 = 0$.
Differentiating \eqref{phiHMDef} shows that for a.e. $r$
\begin{align*}
\phi^\prime\pr{r; \vec v} 
= \pr{2 - N} \frac{1}{r^{N-1}} \int_{B\pr{0, r}} \abs{D \vec v}^2 \d{y} 
+ \frac{1}{r^{N-2}} \int_{\del B\pr{0, r}} \abs{D \vec v}^2 \d{S}\pr{y}.
\end{align*}
Let $\zeta_h\pr{y} = \mu_h\pr{\abs{y}} y$, where
$$\mu_h\pr{s} = \left\{\begin{array}{ll}
1 & s \le r \\
1 + \frac{r-s}{h} & r \le s \le r + h \\
0 & r \ge r+ h
\end{array}  \right..$$
Plugging $\zeta_h$ into \eqref{vectorIdiH} with $\vec v$ and letting $h \to 0^+$, we see that for a.e. $r$
\begin{align}
&\pr{2-N}\int_{B\pr{0,r}} \abs{D \vec v}^2 \d{y} 
= \frac{2}{r } \int_{\del B\pr{0,r}} \abs{y \cdot D \vec v }^2 \d{S}\pr{y} \nonumber \\
&- r \int_{\del B\pr{r,0}} \abs{D \vec v}^2 \d{S}\pr{y}
- \int_{B\pr{0, r}} \vec H \cdot \pr{y \cdot D \vec v} \, \d{y}. 
\label{DvBall}
\end{align}
Substituting this equality into the expression for $\phi^\prime\pr{r; \vec v} $ gives
\begin{align*}
\phi^\prime\pr{r; \vec v} 
&= \frac{1}{r^{N-1}} \brac{\frac{2}{r } \int_{\del B\pr{0,r}} \abs{y \cdot D \vec v }^2 \d{S}\pr{y}
- r \int_{\del B\pr{r,0}} \abs{D \vec v}^2 \d{S}\pr{y}} \\
&+ \frac{1}{r^{N-1}} \brac{
- \int_{B\pr{0, r}} \vec H \cdot \pr{y \cdot D \vec v} \, \d{y}
+ r \int_{\del B\pr{0, r}} \abs{D \vec v}^2 \d{S}\pr{y}} \\
&= \frac{2}{r^N } \int_{\del B\pr{0,r}} \abs{y \cdot D \vec v }^2 \d{S}\pr{y}
- \frac{1}{r^{N-1}} \int_{B\pr{0, r}} \vec H \cdot \pr{y \cdot D \vec v} \, \d{y},
\end{align*}
which proves the corollary since the first term is non-negative.
\end{proof}

To understand that parabolic analogue, we introduce the evolution of harmonic maps.  
The presentation is based on Struwe's paper \cite{St88}; however, it is simplified since we only focus on targets that are spheres.

Let $\vec u : \R^d \times \R \to S^{m-1}$ be a solution to
\begin{equation}
\del_t \vec u + \LP \vec u + \abs{D \vec u}^2 \vec u = 0.
\label{HMEvo}
\end{equation}
We say that a map $\vec u : \R^d \times \R \to \R^m$ is regular iff $\vec u$ and $D\vec u$ are uniformly bounded, and $\disp \der{\vec u}{t}$, $D^2 \vec u$ belong to $L^p_{loc}$ for all $p < \iny$.  
We now have enough notation to state a version of the parabolic analogue of Theorem \ref{HMMono}.
The statement has been reformulated for our purposes.
This result was applied to regularity theory of geometric flows.

\begin{thm}[\cite{St88}, Lemma 3.2]
Let $\vec u : \R^d \times \brac{0, T} \to S^{m-1}$ be a regular solution to \eqref{HMEvo} with $\abs{D\vec u\pr{x,t}} \le c < \iny$ uniformly.  
Let $G_{t}\pr{x}$ be as given in \eqref{GDefn}.
Then the function
$$\Phi\pr{t; \vec u} = t \int_{\R^d} \abs{D\vec u\pr{x}}^2 G_{t}\pr{x} \d{x}$$
is monotonically non-decreasing in $t$.
\end{thm}

\begin{proof}
By assumption, $\vec u : \R^{d} \times \brac{0,T} \to S^{m-1}$. 
For every $n \in \N$, let $\vec v_{n} : B_{T}^n \su \R^{n \cdot d} \to S^{m-1}$ be given by 
$$\vec v_{n}\pr{y} = \vec u\pr{F_{n,d}\pr{y}}.$$
Then by \eqref{Laplace} and \eqref{grSq} from Lemma \ref{ChainR}, we see that
\begin{align*}
&\LP \vec v_n + \abs{D\vec v_n}^2 \vec v_n  \\
&= n\pr{ \LP \vec u + \del_t \vec u + \abs{D \vec u}^2 \vec u } 
+\frac{2}{ d}\brac{\pr{x,t} \cdot \gr_{\pr{x,t}} \pr{\del_t \vec u }  +  {\pr{\pr{x,t} \cdot \gr_{\pr{x,t}} \vec u} \cdot \del_t \vec u \, \vec u}} \\
&= \frac{2}{ d}\brac{\pr{x,t} \cdot \gr_{\pr{x,t}} \pr{\del_t \vec u }  +  {\pr{\pr{x,t} \cdot \gr_{\pr{x,t}} \vec u} \cdot \del_t \vec u \, \vec u}}
=: \vec J
+ \pr{\del_t \vec u \cdot \vec K} \vec u .
\end{align*}
For every $n$, define $\vec H_{n} : B_{T}^n \to \R^{m}$ and $h_{n} : B_{T}^n \to \R$ so that 
\begin{align*}
& \vec H_n\pr{y} = \vec J \pr{F_{n,d}\pr{y}} \\
& h_n\pr{y} = \pr{\del_t \vec u \cdot \vec K} \pr{F_{n,d}\pr{y}}.
\end{align*}
Then, in $B_{T}^n \su \R^{n \cdot d}$
\begin{align}
\LP \vec v_n + \abs{D\vec v_n}^2 \vec v_n 
&= \vec H_n
+h_n \vec v_n .
\label{vnHMiHEq}
\end{align}

For each $n \in \N$, let 
$$\Phi_n\pr{t} = \frac{1}{2 \abs{S^{n \cdot d-1}}} \phi\pr{\sqrt{2dt}; \vec v_n}.$$
Using \eqref{DvBall}, we have
\begin{align*}
&\pr{n \cdot d -2}\phi\pr{\sqrt{2dt}; \vec v_n} 
= 2dt  \int_{S_t^n} \abs{D\vec v_n}^2 \pr{2dt}^{-\frac{n\cdot d -1}{2}} \si_{n\cdot d -1}^t \\
& - 2 \int_{S_t^n} \abs{ y \cdot D \vec v_n }^2 \pr{2dt}^{-\frac{n\cdot d -1}{2}} \si_{n\cdot d -1}^t
+ \pr{2dt}^{-\frac{n\cdot d - 2}{2}}\int_{B_t^n} \vec H_n \cdot \pr{y \cdot D \vec v_n} \, \d{y}.
\end{align*}
Since $\vec v_n\pr{y}\rvert_{S_t^n} = \vec u\pr{f_{n,d}\pr{y}}$, then we may apply Lemma \ref{PFTS} to the first two terms above.  
By \eqref{grSq} from Lemma \ref{ChainR},
\begin{align*}
&\int_{S_t^n} \abs{D \vec v_n}^2 \pr{2dt}^{-\frac{n\cdot d -1}{2}} \si_{n\cdot d -1}^t \\
&= \abs{S^{n \cdot d-1}}  \int_{\R^{d}} \brac{n \abs{D\vec u}^2 
+ \frac{2}{d} \pr{\pr{x,t} \cdot \gr_{\pr{x,t}}\vec u} \cdot  \frac{\del \vec u}{\del t} } G_{t,n}\pr{x} \d{x}.
\end{align*}
And using \eqref{vngr} from Lemma \ref{ChainR},
\begin{align*}
&\int_{S_t^n} \abs{y \cdot D \vec v_n}^2 \pr{2dt}^{-\frac{n\cdot d -1}{2}} \si_{n\cdot d -1}^t 
=  \abs{S^{n \cdot d-1}}  \int_{\R^{d}} \abs{x \cdot \gr \vec u + 2t \frac{\del \vec u }{\del t}}^2 G_{t,n}\pr{x} \d{x}.
\end{align*}
For the third term, we apply Lemma \ref{PFT} and use \eqref{vngr} from Lemma \ref{ChainR} to get
\begin{align}
& \frac{1}{2 \abs{S^{n \cdot d -1}} }\int_{B_t^n} \vec H_n \cdot \pr{y \cdot D \vec v_n} \abs{y}^{\pr{n \cdot d -2}} \abs{y}^{-\pr{n \cdot d -2}} \, \d{y}
\label{HnTerm} \\
&= \int_0^t \int_{\R^d} \pr{\pr{x,\tau} \cdot \gr_{\pr{x,\tau}} \del_\tau \vec u} \cdot \pr{x \cdot \gr \vec u + 2\tau \frac{\del \vec u }{\del \tau}} \pr{2 d \tau}^{\frac{n \cdot d -2}{2}} G_{\tau, n}\pr{x} \, \d{x} \, \d{\tau}.
\nonumber
\end{align}
Therefore,
\begin{align*}
&\pr{nd -2} \Phi_n\pr{t}
=n d t \int_{\R^{d}} \abs{D\vec u}^2 G_{t,n}\pr{x} \d{x} \\
&+\int_{\R^{d}} \brac{2t\pr{\pr{x,t} \cdot \gr_{\pr{x,t}}\vec u} \cdot  \frac{\del \vec u}{\del t} - \abs{x \cdot \gr \vec u + 2t \frac{\del \vec u }{\del t}}^2} G_{t,n}\pr{x} \d{x} \\
&+ \int_0^t \int_{\R^d} \pr{\pr{x,\tau} \cdot \gr_{\pr{x,\tau}} \del_\tau \vec u} \cdot \pr{x \cdot \gr \vec u + 2\tau \frac{\del \vec u }{\del \tau}} \pr{\frac \tau t}^{\frac{n \cdot d -2}{2}} G_{\tau, n}\pr{x} \, \d{x} \, \d{\tau}.
\end{align*}
Taking the limit as $n \to \iny$ and applying Lemma \ref{GnLimit}, we see that
\begin{align}
\lim_{n \to \iny} \Phi_n\pr{t}
&=  t \int_{\R^{d}} \abs{D\vec u}^2G_t\pr{x} \d{x}
=: \Phi\pr{t; u}.
\label{psinLimit}
\end{align}
Since each $\vec v_n$ satisfies \eqref{vnHMiHEq}, then we may apply Corollary \ref{HMMonoiH}.
That is,
\begin{align*}
&\Phi_n^\prime \pr{t} 
\ge \frac{\sqrt{d} }{ \sqrt {8t} \abs{S^{n \cdot d-1}}}  \pr{2 d t}^{- \frac{n \cdot d-1}{2}} \int_{B_t^n} \vec H_n \cdot \pr{y \cdot D \vec v_n}  \, \d{y}  \\
&= \frac{ 1 }{ 2 t }  \int_0^t  \int_{\R^d} \pr{\pr{x,\tau} \cdot \gr_{\pr{x,\tau}} \del_\tau \vec u} \cdot \pr{x \cdot \gr \vec u + 2\tau \frac{\del \vec u }{\del \tau}} \pr{\frac \tau t}^{\frac{n \cdot d -2}{2}}  G_{\tau, n}\pr{x}\, \d{x} \, \d{\tau},
\end{align*}
where we have used \eqref{HnTerm}.
With the use of Theorem \ref{GnLimit}, and that $\disp \lim_{n \to \iny} \pr{\frac \tau t}^{\frac{n \cdot d -2}{2}} = 0$ pointwise for every $\tau \in \pr{0, t}$, the dominated convergence theorem implies that
\begin{align*}
\lim_{n \to \iny } \Phi_n^\prime \pr{t} &\ge 0.
\end{align*}
It follows from this inequality and \eqref{psinLimit} that $\Phi\pr{t; u}$ is monotonically non-decreasing in $t$, proving the theorem.
\end{proof}

\section{Minimal Surfaces and Mean Curvature Flow}
\label{MM}

The theory of minimal surfaces is vast.
We use the following proposition from \cite{CM04} as our definition of a minimal surface.
There are a number of alternate ways to define minimal surfaces, as described in \cite{CM11}, for example.

\begin{prop}
$\Sigma^N \subset \R^k$ is a minimal surface iff the restrictions of the coordinate functions of $\R^k$ to  $\Sigma$ are harmonic functions.
\end{prop}

The next theorem is a monotonicity result for minimal surfaces.
The statement appears in \cite{CM04}, Proposition 4.1 and Lemma 4.2.
This formula is useful in the regularity theory of minimal surfaces.

\begin{thm}
Suppose that $\Sigma^N \subset \R^k$ is a minimal surface and let $w_0 \in \R^k$.  Then the function
\begin{align*}
&\Theta_{w_0}\pr{r; \Sigma} = \frac{\Vol\pr{B_r\pr{w_0} \cap \Sigma}}{\Vol\pr{B_r \subset \R^N}}
\end{align*}
is monotonically non-decreasing in $r$. 
Furthermore,
$$\frac{d}{dr} \Theta_{w_0}\pr{r; \Si} = \frac{N}{\abs{S^{N-1}} r^{N+1}} \int_{\del B_r \cap \Sigma} \frac{\abs{\pr{w-w_0}^\perp}^2}{\abs{\pr{w-w_0}^T}}.$$
\label{MSThm}
\end{thm}

To prove the parabolic analogue, we use the following non-homogeneous version of this result.

\begin{cor}
Let $\Sigma^N \subset \R^{k}$ be a surface with mean curvature denoted by $H$ and outward normal $\bf \nu$.
Assume that $H = h$, for some bounded integrable function $h$.
For the function
\begin{align*}
&\tilde \Theta_{w_0}\pr{r; \Sigma} := \frac{\Vol\pr{B_r\pr{w_0} \cap \Sigma} - \frac 1 N \int_{B_r\pr{w_0} \cap \Si} h \, w \cdot {\bf \nu}}{\Vol\pr{B_r \subset \R^N}}
\end{align*}
we have that
\begin{align*}
&\frac{d}{dr} \tilde \Theta_{w_0}\pr{r; \Si} \\
&= \frac{N}{\abs{S^{N-1}} r^{N+1}}  \int_{\del B_r\pr{w_0} \cap \Sigma} \pr{\frac{\abs{\pr{w-w_0}^\perp}^2}{\abs{\pr{w-w_0}^T}}
- \frac {1}{N} h w \cdot {\bf \nu} \frac{\abs{w - w_0}^2}{\abs{\pr{w - w_0}^T}} }.
\end{align*}
\label{MSCor}
\end{cor}

\begin{proof}
We give the proof in the case where $\Si$ is given by the graph of $v : \R^N \to \R$ so that the coordinates are $w = \pr{y_1, \ldots, y_N, v\pr{y}}$.
Without loss of generality, we may assume that $v\pr{0} = 0$ and take $w_0 = 0$.
We write $B_r$ in place of $B_r\pr{0}$ throughout this proof.
The unit outward normal to $\Si$ is given by
\begin{align}
& {\bf \nu} = \frac{\pr{\gr v, -1}}{\sqrt{1 + \abs{\gr v}^2}},
\label{norVec}
\end{align}
and the first and second fundamental forms are
\begin{align}
 g_{ij} = \de_{ij} + \der{v}{y_i} \der{v}{y_j}, \;\;\;
& g^{ij} = \de_{ij} - \frac{1}{1 + \abs{\gr v}^2} \der{v}{y_i} \der{v}{y_j}, \;\;\;
 \abs{g} = 1 + \abs{\gr v}^2, \nonumber \\
&h_{ij} = \frac{ 1}{\sqrt{1 + \abs{\gr v}^2}}\derm{v}{y_i}{y_j}.
\label{forms}
\end{align}
Therefore, the mean curvature is given by
\begin{equation}
H = g^{ij} h_{ij} = \frac{\LP v}{\sqrt{1 + \abs{\gr v}^2}} - \frac{1}{\pr{1 + \abs{\gr v}^2}^{3/2}} \sum_{i,j =1}^N \der{v}{y_i} \der{v}{y_j} \derm{v}{y_i}{y_j}.
\label{meanCurv}
\end{equation}
The Laplace-Beltrami operator on $\Si$ is 
\begin{align*}
\LP_{\Si} 
&= \frac 1 {\sqrt{\abs{g}}} \sum_{i,j = 1}^N \frac{\del}{\del y_i}\pr{g^{ij} \sqrt{\abs{g}} \der{}{y_j}} \\
&= -\frac{H}{\sqrt{1 + \abs{\gr v}^2}} \sum_{j = 1}^N \der{v}{y_j} \der{}{y_j} 
+\LP
- \frac 1 {{1 + \abs{\gr v}^2}} \sum_{i,j = 1}^N \der{v}{y_i} \der{v}{y_j} \frac{\del^2}{\del y_i\del y_j}.
\end{align*}
Since $H = h$, a computation shows that
\begin{align*}
 \LP_\Si \abs{w}^2
&=2 N - 2 h \, w \cdot {\bf \nu}
\end{align*}
and therefore
\begin{align*}
 N \Vol\pr{B_r \cap \Si} - \int_{B_r \cap \Si} h \, w \cdot {\bf \nu}
&= \frac 1 2 \int_{B_r \cap \Si}  \LP_\Si \abs{w}^2 \\
&= \int_{ B_{v,r}^{N} } \sum_{i,j = 1}^N \frac{\del}{\del y_i}\brac{g^{ij} \sqrt{\abs{g}} \pr{y_j + v \der{v}{y_j}}} \d{y} ,
\end{align*}
where $B_{v,r}^{N} = \set{ y \in \R^N : \abs{y}^2 + \abs{v\pr{y}}^2 \le r^2}$.
Set $S_{v,r}^{N-1} = \set{ y \in \R^N : \abs{y}^2 + \abs{v\pr{y}}^2 = r^2}$, the boundary of $B_{v,r}^{N}$ and let ${\bf{n}} = \frac{y + v \gr v}{\sqrt{\abs{y}^2 + 2 v y \cdot \gr v + v^2 \abs{\gr v}^2}}$, the normal vector to $S_{v,r}^{N-1}$.
An application of the divergence theorem gives
\begin{align*}
& N \Vol\pr{B_r \cap \Si} - \int_{B_r \cap \Si} h \, w \cdot {\bf \nu} \\
&= \int_{ S_{v,r}^{N-1} } \sum_{i,j = 1}^N \pr{ \de_{ij} - \frac{1}{1 + \abs{\gr v}^2} \der{v}{y_i} \der{v}{y_j} } \pr{y_j + v \der{v}{y_j}} \pr{y_i + v \der{v}{y_i}} \times \\
&\times \sqrt{\frac{ 1 + \abs{\gr v}^2 }{\abs{y}^2 + 2 v y \cdot \gr v + v^2 \abs{\gr v}^2}} \d{S}\pr{y} \\
&= \int_{ S_{v,r}^{N-1} } \frac{\abs{y}^2 + 2 v y \cdot \gr v + v^2 \abs{\gr v}^2 + \abs{y}^2 \abs{\gr v}^2 - \pr{y \cdot \gr v}^2}{\sqrt{\pr{1 + \abs{\gr v}^2}\pr{\abs{y}^2 + 2 v y \cdot \gr v + v^2 \abs{\gr v}^2}}}  \d{S}\pr{y} \\
&= \int_{\del B_r \cap \Si} \abs{w^T}.
\end{align*}
By the coarea formula,
\begin{align*}
&\Vol\pr{B_r \cap \Si} - \frac{1}{N} \int_{B_r \cap \Si} h \, w \cdot {\bf \nu} 
=  \int_{B_{v,r}^{N} }\pr{1 - \frac{1}{N} h \, w \cdot {\bf \nu}} \sqrt{1 + \abs{\gr v}^2} \d{y} \\
&=  \int_{-\iny}^{r} \int_{S_{v,\rho}^{N-1} }\pr{1 - \frac{1}{N} h \, w \cdot {\bf \nu}} \sqrt{\frac{\abs{y}^2 + v^2}{\abs{y}^2 + 2 v y \cdot \gr v + v^2 \abs{\gr v}^2 }} \sqrt{1 + \abs{\gr v}^2} \d{S}\pr{y} \d{\rho} ,
\end{align*}
so that
\begin{align*}
&\frac{\d{}}{\d{r}} \brac{r^{-N}\pr{\Vol\pr{B_r \cap \Si} - \frac 1 N \int_{B_r \cap \Si} h \, w \cdot {\bf \nu}} } \\
&= - r^{-N-1 }\brac{ N \Vol\pr{B_r \cap \Si} - \int_{B_r \cap \Si} h \, w \cdot {\bf \nu}}  \\
&+ r^{-N} \frac{\d{}}{\d{r}}\pr{\Vol\pr{B_r \cap \Si} - \frac{1}{N} \int_{B_r \cap \Si} h \, w \cdot {\bf \nu} } \\
&=  r^{-N-1} \int_{S_{v,r}^{N-1}}  \frac{  v^2  - 2 v y \cdot \gr v  + \pr{y \cdot \gr v}^2 }{\sqrt{\pr{1 + \abs{\gr v}^2}\pr{\abs{y}^2 + 2 v y \cdot \gr v + v^2 \abs{\gr v}^2 }}} \d{S}\pr{y} \\
&- \frac 1 N r^{-N+1} \int_{S_{v,r}^{N-1} } \frac{ h\pr{y \cdot \gr v - v} }{ \sqrt{\abs{y}^2 + 2 v y \cdot \gr v + v^2 \abs{\gr v}^2} } \d{S}\pr{y} \\
&=  r^{-N-1} \int_{\del B_r \cap \Si} \pr{ \frac{ \abs{w^\perp}^2 }{ \abs{w^T}} -  \frac 1 N h w \cdot {\bf \nu} \frac{\abs{w}^2}{\abs{w^T}}} ,
\end{align*}
proving the theorem.
\end{proof}

The parabolic analogue of the minimal surface equation is mean curvature flow.  Let $\set{M_t} \subset \R^{d+1}$ be a $1$-parameter family of smooth hypersurfaces.  Then $\set{M_t}$ flows by mean curvature if
\begin{equation}
z_t = \mathbf{H}\pr{z} = \LP_{M_t} z,
\label{MCF}
\end{equation}
where $z$ are the coordinates on $\R^{d+1}$ and $\mathbf{H} = - H \mathbf{\nu}$ denotes the mean curvature vector.  The following theorem is the monotonicity formula due to Huisken \cite{Hu90}.  

\begin{thm}[\cite{Hu90}, Theorem 3.1]
If a smooth $1$-parameter family of hypersurfaces $M_t$ satisfies \eqref{MCF} for $t < T$, then the density ratio 
$$ \vartheta\pr{t; M_t} = \int_{M_t} \pr{T -t}^{-d/2} \exp\pr{-\frac{\abs{z}^2}{4\pr{T-t}}}$$ 
is monotonically non-increasing in $t$.
Furthermore,
\begin{align*}
&\frac{\d{}}{\d{t}}\vartheta\pr{t; M_t} = -\int_{M_t} \abs{\mathbf{H} + \frac{z^\perp}{2\pr{T-t}}}^2 \pr{T-t}^{-d/2} \exp\pr{-\frac{\abs{z}^2}{4\pr{T-t}}}.
\end{align*}
\end{thm}

\begin{proof}
Let $M_t$ be a smooth $1$-parameter family of $d$-dimensional hypersurfaces that flows by mean curvature. 
Assume that each $M_t$ is given by a graph.
Then there exists a function $u : \R^{d} \times \brac{0,T} \to \R$ such that $M_t$ is given in $\R^{d+1}$ by the coordinates
$$\pr{x_1\pr{t}, \ldots, x_d\pr{t}, u\pr{x_1\pr{t}, \ldots, x_d\pr{t}, t}}.$$
Thus, the coordinates of $M_t$ are $z = \pr{x, u\pr{x,t}}$.
The unit outward normal is given by a formula analogous to \eqref{norVec}, while the first and fundamental forms resemble those given in \eqref{forms}.
Therefore, the mean curvature is given by \eqref{meanCurv} with each $v$ replaced by a $u$.
By the assumption that $M_t$ flows by mean curvature, we see that
\begin{align*}
& \pr{{x_1}^\prime\pr{t}, \ldots, {x_d}^\prime\pr{t}, \gr u\pr{x_1\pr{t}, \ldots, x_d\pr{t}, t} \cdot \pr{{x_1}^\prime\pr{t}, \ldots, {x_d}^\prime\pr{t}} +\del_t u} \\
&= -\pr{\frac{\LP u}{\sqrt{1 + \abs{\gr u}^2}}  - \frac{1}{\pr{1 + \abs{\gr u}^2}^{3/2}} \sum_{i,j =1}^d\der{u}{x_i} \der{u}{x_j} \derm{u}{x_i}{x_j} } \frac{\pr{\gr u, -1}}{\sqrt{1 + \abs{\gr u}^2}}.
\end{align*}
Looking at the first $d$ components, we have
\begin{align*}
\pr{{x_1}^\prime\pr{t}, \ldots, {x_d}^\prime\pr{t}} 
&= -\pr{\frac{\LP u}{{1 + \abs{\gr u}^2}} - \frac{1}{\pr{1 + \abs{\gr u}^2}^{2}} \sum_{i,j =1}^d\der{u}{x_i} \der{u}{x_j} \derm{u}{x_i}{x_j}  } \gr u.
\end{align*}
And therefore, from the last component,
\begin{align}
&\der{ u}{t} 
= \LP u - \frac{1}{{1 + \abs{\gr u}^2}} \sum_{i,j =1}^d\der{u}{x_i} \der{u}{x_j} \derm{u}{x_i}{x_j}
= H \sqrt{1 + \abs{\gr u}^2} .
\label{MCFEq}
\end{align}

For each $n \in \N$, let $\Si_n$ be an $n \cdot d$-dimensional hypersurface with coordinates in $\R^{n\cdot d + 1}$ 
$$\pr{y_{1,1}, \ldots, y_{d,n}, \frac{1}{\sqrt{n}}v_n\pr{y_{1,1}, \ldots, y_{d,n}}},$$
where 
\begin{align}
&T- t = \frac{\abs{y}^2}{2d} + \frac{\abs{v_n\pr{y}}^2}{2nd},
\label{tnDefn} \\
& y_{i,1} +  \ldots + y_{i,n} = x_i\pr{t}
\label{xiDefn}
\end{align}
for each $i = 1, \ldots, d$, and each $v_n : B_{v_n/\sqrt{n}, \sqrt{2 d \pr{T-t}}}^{n\cdot d} \su \R^{n\cdot d} \to \R$ is given by
\begin{align*}
v_n\pr{y} &= u\pr{F_{n,d}^{v_n}\pr{y}} \\
&:= u\pr{y_{1,1} +  \ldots + y_{1,n}, \ldots, y_{d,1} +  \ldots + y_{d, n}, T - \frac{\abs{y}^2}{2d} - \frac{\abs{v_n\pr{y}}^2}{2nd}}.
\end{align*}
By analogy with the computations above, the unit outward normal on each $\Si_n$ is
$\disp {\bf \nu}_n = \frac{\pr{\gr v_n, -\sqrt{n}}}{\sqrt{n + \abs{\gr v_n}^2}},$
and the first and second fundamental forms are
\begin{align*}
  g_{n, ij \, kl} = \de_{ik}\de_{jl} + \frac{1}{n}\der{v_n}{y_{i,j}}\der{v_n}{y_{k,l}} \;\;\;\;\;\
& g_n^{ij \, kl} = \de_{ik}\de_{jl} - \frac{1}{n + \abs{\gr v_n}^2}\der{v_n}{y_{i,j}}\der{v_n}{y_{k,l}} \\
\abs{ g_n } = 1 + \frac{1}{n} \abs{\gr v_n}^2 \;\;\;\;\;\
& h_{n, ij \, kl} = \frac{1 }{\sqrt{n + \abs{\gr v_n}^2}}\derm{v_n}{y_{i,j}}{y_{k,l}}.
\end{align*}
Thus,
\begin{align*}
H_n 
&= g_n^{ij \, kl} h_{n, {ij \, kl}} \\
&= \frac{\LP v_n}{\sqrt{n + \abs{\gr v_n}^2}} 
- \frac{1}{\pr{n + \abs{\gr v_n}^2}^{3/2}}\sum_{i,k =1}^d\sum_{j,l=1}^n \der{v_n}{y_{i,j}}\der{v_n}{y_{k,l}} \derm{v_n}{y_{i,j}}{y_{k,l}}.
\end{align*}

As we have changed the definition of $t$, Lemma \ref{ChainR} is not applicable in this setting, so we collect some computations here.
Since
\begin{align*}
\der{v_n}{y_{i,j}}
&= \der{u}{x_i} + \der{u}{t}\pr{-\frac{y_{i,j}}{d} - \frac {v_n}{nd}{\der{v_n}{y_{i,j}}}} 
\end{align*}
then
\begin{align*}
\der{v_n}{y_{i,j}}\pr{1 + \frac{u}{nd}\der{ u}{t} } 
&= { \der{u}{x_i} - \der{u}{t}\frac{y_{i,j}}{d} }.
\end{align*}
Differentiating the first expression with respect to $y_{k, \ell}$ and simplifying gives
\begin{align*}
\derm{v_n}{y_{i,j}}{y_{k,\ell}}\pr{1 + \frac{u}{nd}\der{ u}{t} } 
&= \derm{u}{x_i}{x_k} 
- \derm{u}{x_i}{t} \frac{y_{k,\ell}}{d} 
- \derm{u}{x_k}{t} \frac{y_{i,j}}{d} 
- \der{u}{t}\frac{\de_{i,k} \de_{k,\ell}}{d} 
+ \dert{u}{t} \frac{y_{i,j} y_{k,\ell}}{d^2}\\
&+ \pr{\frac{v_n y_{k,\ell}}{nd^2} \dert{u}{t} - \frac {v_n}{nd} \derm{u}{x_k}{t} }\der{v_n}{y_{i,j}} 
+ \pr{\frac {v_n y_{i,j}}{nd^2} \dert{u}{t} - \frac {v_n}{nd} \derm{u}{x_i}{t}} \der{v_n}{y_{k,\ell}}  \\
&+ \pr{\frac {\abs{v_n}^2}{n^2d^2}\dert{u}{t} -\frac {1}{nd} \der{u}{t}}\der{v_n}{y_{i,j}}\der{v_n}{y_{k,\ell}} .
\end{align*}
Now we multiply through by $\pr{1 + \frac{u}{nd}\der{ u}{t} } ^2$ and use the expression for the first order derivatives from above to simplify the righthandside,
%
\begin{align*}
\derm{v_n}{y_{i,j}}{y_{k,l}} & \pr{1 + \frac{u}{nd}\der{u}{t} }^3 
= \derm{u}{x_i}{x_k}
- \der{u}{t} \frac{\de_{ik} \de_{jl}}{d} 
- \derm{u}{x_k}{t} \frac{y_{i,j}}{d} 
- \derm{u}{x_i}{t} \frac{y_{k,l}}{d}  
+ \dert{u}{t}\frac{y_{i,j} y_{k,l}}{d^2} \\
&+ \frac{1}{nd} A_{ik}
+ \frac{1}{nd^2} B_k y_{i,j}
+ \frac{1}{nd^2} C_i y_{k,l} 
+\frac{1}{nd^3} D y_{i,j} y_{k,l} 
+ \frac{1}{nd^2} E \de_{ik} \de_{jl},
\end{align*}
where each of the terms introduced above depends on $u$ and its derivatives, but is bounded with respect to $n$.
Using the notation $\mathcal{O}_u\pr{1}$ to refer to a term that depends on $u$ and its derivatives, but is bounded with respect to $n$ as $n \to \iny$, we have
\begin{align*}
&\LP v_n \pr{1 + \frac{u}{nd}\del_t u }^3
= n \pr{\LP u - \der{u}{t}}
+ \frac{1}{d} \mathcal{O}_u\pr{1} \\
&\pr{y \cdot \gr v_n }\pr{1 + \frac{u}{nd}\del_t u }
=\pr{x, 2\pr{t-T}} \cdot \gr_{\pr{x,t}} u + \frac{u^2}{n d} \der{u}{t} \\
&\pr{ n + \abs{ \gr v_n}^2 }\pr{1 + \frac{u}{nd}\der{ u}{t} }^{2}
=  n \pr{1 + \abs{\gr u}^2} - \frac{2}{d} \brac{\pr{x,t-T} \cdot \gr_{\pr{x,t}} u - u} \der{u}{t}
\end{align*}
so that
\begin{align*}
&\LP v_n \pr{n + \abs{\gr v_n}^2} \pr{1 + \frac{u}{nd} \der{u}{t}}^5 \\
&= \brac{n\pr{\LP u - \der{u}{t}} + \frac{1}{d}\mathcal{O}_u\pr{1}}
\brac{n\pr{1 + \abs{\gr u}^2} - \frac{2}{d} \brac{ {\pr{x,t-T} \cdot \gr_{\pr{x,t}} u} -u }\der{u}{t}}  \\
&= n^2 \pr{\LP u - \der{u}{t}}\pr{1 + \abs{\gr u}^2 }
+ \frac{n}{d} \mathcal{O}_u\pr{1}, 
\end{align*}
and 
\begin{align*}
&\sum_{i,k =1}^d \sum_{j,l=1}^n \der{v_n}{y_{i,j}}\der{v_n}{y_{k,l}} \derm{v_n}{y_{i,j}}{y_{k,l}} \pr{1 + \frac{u}{nd} \der{u}{t}}^5 \\
&= \sum_{i,k =1}^d \sum_{j,l=1}^n 
\pr{\der{u}{x_i} - \der{u}{t} \frac{y_{i,j}}{d} } 
\pr{\der{u}{x_k} - \der{u}{t} \frac{y_{k,l}}{d} } \times \\
&\times \pr{ \derm{u}{x_i}{x_k}
- \der{u}{t} \frac{\de_{ik} \de_{jl}}{d} 
- \derm{u}{x_k}{t} \frac{y_{i,j}}{d} 
- \derm{u}{x_i}{t} \frac{y_{k,l}}{d}  
+ \dert{u}{t}\frac{y_{i,j} y_{k,l}}{d^2}} \\
&+ \frac{1}{nd} \sum_{i,k =1}^d \sum_{j,l=1}^n 
\pr{\der{u}{x_i} - \frac{y_{i,j}}{d} \der{u}{t} } 
\pr{\der{u}{x_k} - \frac{y_{k,l}}{d} \der{u}{t} } \times \\
& \times \pr{ A_{ik} + \frac{1}{d} B_k y_{i,j} + \frac{1}{d} C_i y_{k,l} + \frac{1}{d^2} D y_{i,j} y_{k,l} + + \frac{1}{d} E \de_{ik} \de_{jl}} \\
&= n^2 \sum_{i,k =1}^d \der{u}{x_i}\der{u}{x_k} \derm{u}{x_i}{x_k} 
+ \frac{n}{d}\mathcal{O}_u\pr{1}.
\end{align*}
Since $u$ satisfies \eqref{MCFEq}, then
\begin{align*}
&\brac{\LP v_n\pr{n + \abs{\gr v_n}^2}  - \sum_{i,k =1}^d\sum_{j,l=1}^n \der{v_n}{y_{i,j}}\der{v_n}{y_{k,l}} \derm{v_n}{y_{i,j}}{y_{k,l}} } \pr{1 + \frac{u}{nd} \der{u}{t}}^5 \\
&=  n^2 \brac{\pr{\LP u - \der{u}{t}}\pr{1 + \abs{\gr u}^2 } - \sum_{i,k =1}^d \der{u}{x_i}\der{u}{x_k} \derm{u}{x_i}{x_k} }+ \frac{n}{d} \mathcal{O}_u\pr{1} \\
&= \frac{n}{d} \mathcal{O}_u\pr{1}.
\end{align*}
Therefore, $H_n = h_n$ where $H_n$ is the mean curvature of $\Si_n$ and $h_n : B_{v_n/\sqrt{n}, \sqrt{2 d \pr{T-t}}}^{n\cdot d} \to \R$ satisfies
\begin{align*}
h_n\pr{F_{n,d}\pr{y}} 
&= \frac{ \La_n }{ \sqrt n \, d \pr{1 + \abs{\gr u}^2 - \frac{2}{n d} \brac{ {\pr{x,t-T} \cdot \gr_{\pr{x,t}} u - u} }\der{u}{t}}^{3/2}\pr{1 + \frac{u}{nd} \der{u}{t}}^2},
\end{align*}
where $\La_n = \mathcal{O}_u\pr{1}$.
We may apply Corollary \ref{MSCor} to $\Si_n$ with $r = \sqrt{2 d \pr{T-t}}$ for $t < T$.
As shown in the proof of Corollary \ref{MSCor},
\begin{align*}
& \abs{S^{n \cdot d -1}} \pr{2 d \pr{T-t}}^{\frac{n \cdot d}{2}} \tilde \Theta_0\pr{\sqrt{2 d \pr{T-t}}; \Si_n} \\
&= { n \cdot d \Vol\pr{B_{\sqrt{2 d \pr{T-t}}} \cap \Si_n}
- \int_{\Si_n \cap \set{\abs{w} \le \sqrt{2 d \pr{T-t}}} } h_n w \cdot {\bf \nu}_n} \\
&= \int_{S^{n \cdot d -1}_{v_n, t}} \frac{\abs{y}^2 + \frac{2}{n} v_n y \cdot \gr v_n + \frac{1}{n^2} v_n^2 \abs{\gr v_n}^2 + \frac{1}n \abs{y}^2 \abs{\gr v_n}^2 - \frac 1 n \pr{y \cdot \gr v_n}^2}{\sqrt{\pr{1 + \frac 1 n \abs{\gr v_n}^2}\pr{\abs{y}^2 + \frac 2 n v_n y \cdot \gr v_n + \frac 1 {n^2}v_n^2 \abs{\gr v_n}^2}}} \si_{n \cdot d}^{v_n, t},
\end{align*}
where $\si_{n \cdot d}^{v_n, t}$ is the surface measure of $S^{n \cdot d -1}_{v_n, t} := \set{y \in \R^{n \cdot d} : \abs{y}^2 + \frac{1}{n} v_n\pr{y}^2 = 2 d \pr{T-t} }$.

As $n \to \iny$, $S^{n \cdot d -1}_{v_n, t}$ will approach the sphere is radius $\sqrt{2d\pr{T-t}}$ in $\R^{n \cdot d}$.
However, since it is not in fact a sphere, and the time direction is different from the previous examples, we must repeat the arguments from Section \ref{MT} for this new setting.

We define
$$f_{n,d}^{v_n} : \R^{n \cdot d} \to \R^d$$
so that \eqref{xiDefn} is satisfied for each $i = 1, \ldots, d$.
We can then use $f_{n,d}^{v_n}$ to map the coordinates of $S^{n \cdot d -1}_{v_n, t}$ to coordinates in $M_t$.
In fact, $f_{n,d}^{v_n}$ maps $S^{n \cdot d -1}_{v_n, t}$ to the set $B_{u, nt} := \set{x \in \R^d : \abs{x\pr{t}}^2 + \abs{u\pr{x, t}}^2 \le 2 n d \pr{T-t}}$.

By analogy with the computations given in Section \ref{MT}, if $\vp: \R^d \to \R$ is integrable with respect to $G_{t,n}^u\pr{x} \d{x}$, then
\begin{align}
& \frac{1}{\abs{S_{v_n, t}^{n \cdot d -1}}} \int_{S^{n \cdot d -1}_{v_n, t}} \vp\pr{f_{n,d}^{v_n}\pr{y}} \si_{n \cdot d-1}^{v_n, t}
= \int_{\R^d} \vp\pr{x} G_{t,n}^u\pr{x} \d{x}.
\label{PFTSvn}
\end{align}
where $G_{t, n}^u\pr{x}$ is a measure supported on $B_{u, nt}$ with the property that
\begin{equation}
\lim_{n \to \iny} G_{t, n}^u\pr{x} = \frac{1}{C_{d}} \pr{T-t}^{-\frac{d}{2}} \exp\pr{- \frac{\abs{x}^2 + u^2}{4\pr{T-t}}}.
\label{GnuDefn}
\end{equation}

Set
$$ \Phi_n\pr{t} = C_d \tilde \Theta_0\pr{\sqrt{2d\pr{T-t}}; \Si_n}$$
so that
\begin{align*}
& \frac{\abs{S^{n \cdot d -1}} \pr{2 d \pr{T-t}}^{\frac{n \cdot d}{2}} }{C_d}  \Phi_n\pr{t} 
= \int_{S^{n \cdot d -1}_{v_n, t}} \abs{y}^2 \sqrt{\frac{{1 + \frac{1}n \abs{\gr v_n}^2} }{{\abs{y}^2 + \frac 2 n v_n y \cdot \gr v_n + \frac 1 {n^2}v_n^2 \abs{\gr v_n}^2}}} \si_{n \cdot d}^{v_n, t} \\
&+ \int_{S^{n \cdot d -1}_{v_n, t}} \frac{\frac{1}{n}\brac{2 v_n y \cdot \gr v_n + \frac{1}{n} v_n^2  \abs{\gr v_n}^2  - \pr{y \cdot \gr v_n}^2}}{\sqrt{\pr{1 + \frac 1 n \abs{\gr v_n}^2}\pr{\abs{y}^2 + \frac 2 n v_n y \cdot \gr v_n + \frac 1 {n^2}v_n^2 \abs{\gr v_n}^2}}} \si_{n \cdot d}^{v_n, t}.
\end{align*}
It follows from \eqref{PFTSvn} that
\begin{align*}
 \frac{\abs{S^{n \cdot d -1}} \pr{2 d \pr{T-t}}^{\frac{n \cdot d -1}{2}}} {C_d \abs{S_{v_n, t}^{n \cdot d -1}}}\Phi_n\pr{t} 
 &= \int_{\R^d} \sqrt{1 + \abs{\gr u}^2}\pr{1 + \mathcal{O}_u\pr{\frac 1 n}} G_{t,n}^u\pr{x} \d{x}
\end{align*}
and therefore,
\begin{align}
\lim_{n \to \iny} \Phi_n\pr{t} 
 &= \int_{\R^d} \pr{T- t}^{-\frac{d}{2}} \exp\pr{- \frac{\abs{x}^2 + u^2}{4\pr{T-t}}} \sqrt{1 + \abs{\gr u}^2} \d{x} \nonumber \\
& = \int_{M_t} \pr{T-t}^{-\frac{d}{2}} \exp\pr{- \frac{\abs{z}^2}{4\pr{T-t}}}  
 =  \vartheta\pr{t; M_t}.
 \label{ThetaLim}
\end{align}
By the computations from Corollary \ref{MSCor},
\begin{align*}
&\frac{\abs{S^{n \cdot d -1}} \pr{2 d \pr{T-t}}^{\frac{n \cdot d }{2}}}{d} \Phi_n^\prime\pr{t}
= \frac{C_d}{\sqrt n}\int_{S^{n \cdot d -1}_{v_n, t}}  \frac{ h_n \pr{v_n - y \cdot \gr v_n} }{ \sqrt{\abs{y}^2 + \frac 2 n v_n y \cdot \gr v_n + \frac 1 n v_n^2 \frac 1 n \abs{\gr v_n}^2} } \si_{n \cdot d}^{v_n, t} \\
&- \frac{C_d}{2 \pr{T-t} } \int_{S^{n \cdot d -1}_{v_n, t}} \frac{ \pr{v_n- y \cdot \gr v_n}^2 }{\sqrt{\pr{1 + \frac 1 n \abs{\gr v_n}^2}\pr{\abs{y}^2 + \frac 2 n v_n y \cdot \gr v_n + \frac 1 n v_n^2 \frac 1 n \abs{\gr v_n}^2 }}} \si_{n \cdot d}^{v_n, t} .
\end{align*}
Using the expressions from above along with equation \eqref{PFTSvn}, we have
\begin{align*}
& \frac{\abs{S^{n \cdot d -1}} \pr{2 d \pr{T-t}}^{\frac{n \cdot d -1}{2}}}{C_d \abs{S_{v_n, t}^{n \cdot d -1}}} \Phi_n^\prime\pr{t} \\
&=  \frac{1}{2 n d \pr{T-t} } \int_{\R^d} \frac{ \La_n \brac{u- \pr{x, 2\pr{t-T}} \cdot \gr_{\pr{x,t}} u}  }{ \pr{1 + \abs{\gr u}^2}^{\frac 3 2}} \pr{1 + \mathcal{O}_u\pr{\frac 1 n}}  G_{t,n}^u\pr{x} \d{x} \\
&- \frac{1}{4 \pr{T-t}^2} \int_{\R^d} \frac{ \brac{u- \pr{x, 2\pr{t-T}} \cdot \gr_{\pr{x,t}} u}^2 }{ 1 +  \abs{\gr u}^2 } \sqrt{ 1 +  \abs{\gr u}^2 } \pr{1 + \mathcal{O}_u\pr{\frac 1 n}} G_{t,n}^u\pr{x} \d{x}
\end{align*}
so that
\begin{align*}
\lim_{n \to \iny} \Phi_n^\prime\pr{t}
&= - \int_{\R^d} \frac{ \brac{\frac{x \cdot \gr u - u}{2\pr{T-t}} - \der{u}{t} }^2 }{ 1 +  \abs{\gr u}^2 } \sqrt{ 1 +  \abs{\gr u}^2 } \pr{T- t}^{-\frac{d}{2}} \exp\pr{- \frac{\abs{x}^2 + u^2}{4\pr{T-t}}} \d{x}.
\end{align*}
Since ${\bf H} = - H {\bf \nu}$ and $z^\perp = \pr{z \cdot {\bf \nu}} \nu$, then
\begin{align*}
{\bf H} + \frac{z^\perp}{2\pr{T-t}}
&= \pr{\frac{ z \cdot {\bf \nu} }{2\pr{T-t}} - H } {\bf \nu}
= \pr{ \frac{ x \cdot \gr u - u }{2\pr{T-t}} -  \der{u}{t}} \frac{\bf \nu}{\sqrt{1 + \abs{\gr u}^2}},
\end{align*}
where we have used \eqref{MCFEq} and \eqref{norVec}.
Since $\bf \nu$ has unit length, then
\begin{align*}
\lim_{n \to \iny} \Phi_n^\prime\pr{t}
&= \int_{M_t} \abs{{\bf H} + \frac{z^\perp}{2\pr{T-t}}}^2 \pr{T-t}^{-\frac{d}{2}} \exp\pr{- \frac{\abs{z}^2}{4\pr{T-t}}}.
\end{align*}
In combination with \eqref{ThetaLim}, the conclusion of the theorem follows.
\end{proof}

\noindent
{\small {\bf Acknowledgements}
This project was started at the University of Minnesota where the author worked as a postdoc.
The author is extremely grateful to her postdoctoral mentor, Vladimir Sverak, for suggesting this project and for his support and guidance.
The author would also like to express her gratitude to Luis Escauriaza for his helpful comments and for suggesting a section devoted to the two-phase monotonicity formulas.
This project was supported in part by PSC-CUNY Award \#69078-00 47.}

\noindent
{\small {\bf Conflict of Interest}
 The author declares that they have no conflict of interest.}

\end{document}